\documentclass{amsart}
\usepackage{amssymb,amsmath,amsthm,amsfonts,color}
\usepackage{mathrsfs,dsfont}
\usepackage{graphicx,graphics}
\usepackage{esint}
\usepackage[applemac]{inputenc}
\usepackage{yhmath}

\numberwithin{equation}{section}
\theoremstyle{plain}
\newtheorem{theorem}{Theorem}[section]

\newtheorem{lemma}[theorem]{Lemma}

\newtheorem{corollary}[theorem]{Corollary}

\theoremstyle{definition}

\newtheorem{remark}[theorem]{Remark}

\newcommand\h{\mathscr H^{1}}

\newcommand\ds\displaystyle

\newcommand\ph{\varphi}
\newcommand\R{\mathbb R}

\newcommand\T{\mathscr T_2}

\renewcommand\L{\mathcal L}
\newcommand\D{\mathscr D}

\newcommand\Y{{Y_2}}

\newcommand\Lr{\mathscr L}
\newcommand\LL{\mathbb L}
\newcommand\KK{\mathbb K}
\newcommand\Ks{\mathscr K}
\newcommand\Z{\mathbb Z}
\newcommand\ZZ{\mathscr Z}
\newcommand\esym{{\mathbf E}}

\renewcommand\ell{p}

\newcommand\bl{{\bar\lambda}}
\newcommand\bm{{\bar\mu}}

\newcommand\essinf{\mathop{\operatorname{ess-inf}}}

\renewcommand\div{\mbox{\rm div\,}}
\newcommand\tr{\mbox{\rm tr\,}}

\newcommand{\Om}{\Omega}
\newcommand\MS{\R^{2\times 2}_s}

\newcommand\iO{\int_\Omega}
\newcommand\iY{\int_{\Y}}

\newcommand\iYt{\int_{Y_2}}
\newcommand\iR{\int_{\R^2}}
\newcommand\itwo[1]{\int_{\Om\times Z_{#1}}}
\newcommand\itwot{\int_{\Om\times \Y}}

\newcommand\e{\varepsilon}
\newcommand\ep{\varepsilon}
\newcommand\ue{u^\e}
\newcommand{\LLper}{\Lambda_{\mbox{\rm\tiny per}}}

\newcommand\avse{\alpha_{\mbox{\rm\tiny vse}}}
\newcommand\ase{\alpha_{\mbox{\rm\tiny se}}}

\newcommand\mn{\mathbb{R}^{2\times 2}}
\newcommand\mtwo{\mathbb{R}^{2\times 2}}

\newcommand\Csc{C^\infty_c(\Om;\R^2)}
\newcommand\Cinf{C^\infty_c(\Om\times \T;\mn)}

\newcommand\OR{\Om;\R^2}
\newcommand\ORR{\Om;\mn}
\newcommand\OYR{\Om\times \T;\R^2}
\newcommand\OYRR{\Om\times \Y;\mn}
\newcommand\YR{\Y;\R^2}
\newcommand\Omb{\bar\Omega}

\newcommand\XX{{\mathbf X}}

\newcommand\wk{\rightharpoonup}
\newcommand\wkst{\stackrel{\star}{\rightharpoonup}}
\newcommand\wkg{\stackrel{\Gamma(L^2)}{\rightharpoonup}}

\newcommand\wkgh{\stackrel{\Gamma(H^1_0)}{\rightharpoonup}}

\newcommand\be[1]{\begin{equation}\label{#1}}
\newcommand\ee{\end{equation}}
\newcommand\ba[1]{\left\{\begin{array}{#1}}
\newcommand\ea{\end{array}\right.}
\newcommand\twosc{\begin{array}{c}\rightharpoonup\\[-.3cm]\rightharpoonup\end{array}}

\newcommand\la{\lambda}
\renewcommand\sharp{{\rm per}}

\title [2D aether  and homogenization]{A two-dimensional labile aether through homogenization}
\author[M. Briane] {Marc Briane} 
\address[Marc Briane]{Univ Rennes, INSA Rennes,  CNRS, IRMAR - UMR 6625, F-35000 Rennes, France}
\email[M. Briane]{mbriane@insa-rennes.fr}
\author[G. Francfort] {Gilles A. Francfort} 
\address[Gilles Francfort]{LAGA, Universit\'e Paris-Nord \& Courant Institute of Mathematical Sciences, New York University }
\email[G. Francfort]{gilles.francfort@univ-paris13.fr}
\date{\today}

\begin{document}
\vskip .2truecm
\begin{abstract}
\small{Homogenization in linear elliptic problems usually assumes  coercivity of the accompanying Dirichlet form. In linear elasticity, coercivity is not ensured through mere (strong) ellipticity so that the usual estimates that render homogenization meaningful break down unless stronger assumptions, like very strong ellipticity, are put into place. Here, we demonstrate that a $L^2$-type homogenization process can still be performed, very strong ellipticity notwithstanding, for a   specific two-phase two dimensional problem whose significance derives from prior work establishing that  one can lose strong ellipticity in such a setting, {\it provided} that homogenization turns out to be meaningful.

A striking consequence is that, in an elasto-dynamic setting, some two-phase homogenized laminate may support plane wave propagation in the direction of lamination on a bounded domain with Dirichlet boundary conditions, a possibility which does not exist for  the associated two-phase microstructure at a fixed scale. Also, that material blocks longitudinal waves in the direction of lamination, thereby acting as a two-dimensional aether in the sense of {\em e.g.} Cauchy.
}
\end{abstract}
 
\maketitle

\section{Introduction}

This paper may be viewed as a sequel to both \cite{briane.francfort2015} and  \cite{francfort.gloria}. Those, in turn,  were a two-dimensional revisiting of \cite{geymonat.muller.triantafyllidis93} in the light of \cite{gutierrez99}.
The issue at stake was whether one could lose strict strong ellipticity when performing a homogenization process on a periodic mixture of two isotropic elastic materials, one being (strictly) very strongly elliptic while the other is only (strictly) strongly elliptic. We start this introduction with a brief overview of the problem that had been addressed in those papers, restricting all considerations to the two-dimensional case.

 We  consider throughout an elasticity tensor (Hooke's law) of the form
 $$
\LL\in L^\infty\big(\T; \Lr_s(\MS)\big),
 $$
 where $\T$ is the $2$-dimensional torus $\R^2/\mathbb Z^2$ and $ \Lr_s(\MS)$ denotes the set of symmetric mappings from the set of $2\!\times\!2$ symmetric matrices onto itself. Note that there is a canonical identification $\mathcal I$  between $\T$ and the unit cell  $\Y:=[0,1)^2$; for simplicity, we will denote by $y$ both an element of $\T$ and its image under the mapping $\mathcal I$.  
 
 The tensor-valued function $\LL$ defined in $\T$ is  extended by $\Y$-periodicity to $\R^2$ as 
\[
\LL(y+\kappa)=\LL(y),\quad\mbox{a.e. in }\R^2,\ \forall\,\kappa\in\Z^2,
\]
so that the rescaled function $\LL(x/\e)$ is $\e \Y$-periodic.
\par
We then consider the Dirichlet boundary value problem on a bounded open domain $\Omega\subset\R^2$
\be{eq.Dir}
\ba{cll}
-\,\div\big(\LL(x/\e)\nabla \ue\big) &=  & f\mbox{ in }\Om
\\[2mm]
\ue&= & 0\mbox{ on }\partial \Om,
\ea
\ee
with $f\in H^{-1}(\OR)$.
We could impose a very strong ellipticity condition on $\LL$, namely
\be{def.alp-ver-str-ell}
\avse(\LL):=\essinf_{y\in\T}\left(\min\big\{\LL(y)M\cdot M:M\in \MS, |M|=1\big\}\right)>0.
\ee
In such a setting, homogenization is straightforward; see {\em e.g.} the remarks in  \cite[Ch. 6, Sec. 11]{sanchez80}. 
\par
Instead, we will merely impose (strict) strong ellipticity, that is
\be{def.alp-str-ell}
\ase(\LL):=\essinf_{y\in\T}\left(\min\big\{\LL(y)(a\otimes b)\cdot (a\otimes b):a,b\in \R^2, |a|=|b|=1\big\}\right)>0,
\ee
and this throughout.
\begin{remark}[Ellipticity and isotropy]
\label{rem.iso-str-ell}
Whenever $\LL$ is isotropic, that is
\[
\LL(y)M=\lambda(y)\,\tr(M)\,I_2+2\mu(y)\,M,\quad\mbox{for }y\in\T,\ M\in\R_s^{2\times 2},
\]
then \eqref{def.alp-ver-str-ell} reads as
$$
\essinf_{y\in\T}\left(\min\big\{\mu(y),\lambda(y)+\mu(y)\big\}\right)>0
$$
while \eqref{def.alp-str-ell} reads as
$$
\essinf_{y\in\T}\left(\min\big\{\mu(y),\lambda(y)+2\mu(y)\big\}\right)>0.
$$
\vskip-.8cm\hfill\P\end{remark}

The strong ellipticity condition \eqref{def.alp-str-ell} is the starting point of the study of homogenization performed in \cite{geymonat.muller.triantafyllidis93} from a variational standpoint, that of $\Gamma$-convergence. Under that condition, the authors investigate
the 
 $\Gamma$-convergence, for the weak topology of $H^1_0(\OR)$ on bounded sets (a metrizable topology), of the Dirichlet integral 
$$\iO \LL(x/\e)\nabla v\cdot\nabla v\,dx.$$
Then, under certain conditions that will be recalled in Section \ref{kr}, the $\Gamma$-limit is given through the expected homogenization formula 
\be{form-hom-cl}
\LL^0M\cdot M:= \min\left\{\iY \LL(y)(M+\nabla v)\cdot (M+\nabla v)\,dy:v\in H^1_{\rm per}(\YR)\right\}
\ee
 in spite of the lack of very strong ellipticity.
 
 In \cite{gutierrez99,gutierrez04}, the viewpoint is somewhat different. The author,  S. Guti\'errez, looks at a two-phase layering of a very strongly elliptic isotropic material with a strongly elliptic isotropic material. Assuming that the homogenization process makes sense, he shows that strict strong ellipticity can be lost through that process for a very specific combination of Lam\'e coefficients (see \eqref{hyp-lam-mu} below) and for a volume fraction $1/2$ of each phase. 
 
 Our goal in the previous study \cite{briane.francfort2015} was to reconcile those two sets of results, or more precisely, to demonstrate that {Guti\'errez}' viewpoint expounded in \cite{gutierrez99,gutierrez04} fit within the variational framework set forth in \cite{geymonat.muller.triantafyllidis93} and that the example produced in those papers is the only possible one within the class of laminate-like microstructures.  Then, it is shown in \cite{francfort.gloria} that the {Guti\'errez} pathology is in essence canonical, that is that inclusion-type microstructures never give rise to such a pathology.  
 
 The concatenation of those results may be seen as an indictment of linear elasticity, especially when confronted with its scalar analogue where ellipticity cannot be weakened through a homogenization process. However, our results, hence those of Guti\'errez, had to be tempered by the realization that $\Gamma$-convergence {\it a priori} assumes convergence of the relevant sequences in the {\it ad hoc} topology (here the weak-topology on bounded sets of $H^1_0$). The derivation of a bound that allows for such an assumption not to be vacuous is not part of the $\Gamma$-convergence  process, yet it is essential lest that process become a gratuitous mathematical exercise.
 
 This is the primary task that we propose to undertake in this study. To this end we add to the Dirichlet integral
 a zeroth-order term of the form $\int_\Omega |v|^2\,dx$ which will  immediately provide compactness in the weak topology of $L^2(\Omega;\R^2)$. We are then led to an investigation, for the {\it weak topology on bounded sets of}  $L^2(\Omega;\R^2)$, of the $\Gamma$-limit of the Dirichlet integral
 $$\iO \LL(x/\e)\nabla v\cdot\nabla v\,dx.$$
 
 {Our ``elliptic" results, detailed in Theorems \ref{thm.main}, \ref{thm.mainbis}, essentially state} that, at least for periodic mixtures of two isotropic materials that satisfy the constraints imposed in \cite{gutierrez99},  the ensuing $\Gamma$-limit is in essence identical to that which had been previously obtained  for the {\it weak topology on bounded sets of}  $H^1_0(\Omega;\R^2)$. An immediate consequence is that {Guti\'errez}' example does provide a {\it bona fide} loss of strict strong ellipticity in two-phase two-dimensional periodic homogenization, and not only one that would be conditioned upon some otherwordly bound  on minimizing sequences;  see Lemma~\ref{lem.ow}. 
 
We then move on to the hyperbolic setting and demonstrate that the results of Theorems \ref{thm.main}, \ref{thm.mainbis} imply  a weak homogenization result for the equations of elasto-dynamics which leads, in the Gutiérrez example, to the striking and, to the best of our knowledge, new realization that homogenization may lead to  a plane wave propagation for the homogenized system on a bounded domain with Dirichlet boundary conditions, although, at a fixed scale, the microstructure would of course prevent  such a  propagation, precisely because of  the Dirichlet boundary condition. This is roughly because a degeneracy of strong ellipticity in some direction relaxes the boundary condition on a certain part of the boundary. 

Further, the Gutiérrez material is unique in its anisotropy class (2D orthorombic) in blocking  longitudinal waves -- those  for which propagation and oscillation are in the same direction --  in a some preset direction.  This feature motivates the title of our contribution because such a property was precisely the focus  of  pre-Maxwellian investigations by, among others, Cauchy, Green, Thomson (Lord Kelvin). There, an elastic substance called {\em labile aether} was meant to carry light throughout space, thereby spatially co-existing with the various materials it permeated \cite[Chapter~5]{whittaker}.  In order to  conform to the various available observations   for the   propagation of light, it was deemed imperative that aether, as an elastic material, should allow for transverse plane waves while inhibiting longitudinal waves. According to \cite{whittaker},  Green's 1837 theory of wave reflection for elastic solids that  assumed, in Fresnel's footstep,  that
aether should be much stiffer in compression than in shear prompted Cauchy's 1839 publication of his  third theory of reflection in a material for which the Lamé coefficients $\lambda, \mu$ satisfy 
\be{zero.ws}
\la + 2\mu = 0.
\ee
This is precisely what the Gutiérrez material achieves, at least in a crystalline way, by forbidding longitudinal waves in the direction of lamination.

In Section \ref{kr}, we provide a quick review of the results that are relevant to our investigation. Then Section~\ref{sec.gam} details the precise assumptions under which we obtain {Theorems \ref{thm.main}, \ref{thm.mainbis} and present the proofs of those theorems.  Section~\ref{concl}  details the impact of our  results on the actual minimization of the above mentioned Dirichlet integrals augmented by a linear (force) term. Such a minimization process provides in turn a homogenization result for elasto-dynamics (Theorem \ref{thm.elastod}) in a setting where strict strong ellipticity is lost in the limit. We conclude with a  discussion of the propagation properties of the Gutiérrez  material.
 
Throughout the paper, the following remark will play a decisive role. Since, for $v\in H^1_0(\Omega;\R^2)$, the mapping  $v\mapsto\det\left(\nabla v\right)$ is a null Lagrangian, we are at liberty to replace the Dirichlet integral under investigation by
\[
 \iO \big\{\LL(x/\e)\nabla v\cdot\nabla v+c\det\left(\nabla v\right)\big\}\,dx,
\]
for any $c\in \R$, thereby replacing $$M\mapsto \LL(y)M, \; M\in \R^{2\times 2},$$ by $$M\mapsto \LL(y)M+\frac c2 \;{\rm cof}\left(M\right),\; M\in \R^{2\times 2}.$$

\par\bigskip\noindent
Notationwise,
\begin{itemize}
\item[-]  $I_2$ {is} the unit matrix of $\R^{2\times 2}$; $R^\perp$ is the $\pi/2$-rotation matrix 
$\left(\begin{smallmatrix}  0&-1\\1&\;\ 0\end{smallmatrix}\right)$;
\item[-]  $A\cdot B$ {is} the Frobenius inner product between two elements of $A,B\in\mn$, that is $A\cdot B:=\tr(A^TB)$;
\item[-] If $A:=\left(\begin{smallmatrix} a&\; c\\b&\; d\end{smallmatrix}\right)\in \mtwo,$ the cofactor matrix of $A$ is
${\rm cof}\left(A\right):=\left(\begin{smallmatrix} \;\ d&-b\\-c&\;\ a\end{smallmatrix}\right)$;
\item[-] If $\KK:\R^p\to\R^p$ is a linear mapping, the {\em pseudo-inverse} of $\KK$, denoted by~$\KK^{-1}$, is defined on its range ${\rm Im}(\KK)$ as follows:  for any $\xi\in\R^p$, $\KK^{-1}(\KK\,\xi)$ is the orthogonal projection of $\xi$ onto the orthogonal space $[{\rm Ker}\,(\KK)]^\perp$, so that $\KK\big(\KK^{-1}(\KK\,\xi)\big)=\KK\,\xi$;
\item[-] {If $u$ is a distribution (an element of $\D'(\R^2;\R^2)$), then $$\ds{\rm curl\ }u:=\frac{\partial u_1}{\partial x_2}-\frac{\partial u_2}{\partial x_1}$$ while $$\esym(u)=\left(\begin{smallmatrix} \frac{\partial u_1}{\partial x_1}&\frac12\left(\frac{\partial u_1}{\partial x_2}+\frac{\partial u_2}{\partial x_1}\right)\\\frac12\left(\frac{\partial u_1}{\partial x_2}+\frac{\partial u_2}{\partial x_1}\right)&\frac{\partial u_2}{\partial x_2}\end{smallmatrix}\right);$$}
\item[-] $H^1_\sharp(Y_2;\R^p)$ (resp. $L^2_\sharp(Y_2;\R^p), L^\infty_\sharp(Y_2;\R^p), C^p_\sharp(Y_2;\R^p)$) is the space of those functions in $H^1_{\rm loc}(\R^2;\R^p)$ (resp. $L^2_{\rm loc}(\R^2;\R^p), \!L^\infty(\R^2;\R^p), \!C^p(\R^2;\R^p)$) that are $Y_2$-periodic;
{\item[-] For any subset $\ZZ\in \T$, we agree to denote by $Z$ its representative in $\Y$ through the canonical representation $\mathcal I$ introduced earlier, and by $Z^\#$ its representative in $\R^2$, that is the open ``periodic" set $$Z^\#:=\mathring{\wideparen{\bigcup_{k\in \Z^2}(k+\overline{Z})}}.$$}
\item[-] Throughout, the variable $x$ will refer to a running point in a {\em bounded} open domain $\Omega\subset\R^2$, while the variable $y$ will refer to a running point in $\Y$ (or~$\T$, {or $k+\Y,\; k\in \mathbb Z^2$});
\item[-] {If ${\mathscr I}^\e$ is an $\e$-indexed sequence of functionals with
$${\mathscr I}^\e: X\to\R,$$
(X reflexive Banach space), we will write that ${\mathscr I}^\e\stackrel{\Gamma(X)}{\wk} {\mathscr I}^0$, with
$${\mathscr I}^0:X\to \R,$$
if ${\mathscr I}^\e \;\Gamma$-converges to ${\mathscr I}^0$ for the weak topology on bounded sets of $X$} (see {\em e.g.} \cite{dalmaso93} for the appropriate definition); and 
\item[-] $u^\e\twosc u^0$ where $u^\e\in L^2(\OR)$ and $u^0\in L^2(\OYR)$ iff $u^\e$ two-scale converges to $u^0$ in the sense of Nguetseng; see {\em e.g.} \cite{nguetseng89,allaire92}. \end{itemize}

\section{Known results}\label{kr}

As previously announced, this short section recalls the relevant results obtained in \cite{geymonat.muller.triantafyllidis93}, \cite{briane.pallares}.  For vector-valued (linear) problems,  a successful application of Lax-Milgram's lemma to a Dirichlet problem of the type \eqref{eq.Dir} hinges on the positivity of the following functional coercivity constant:
\[
\Lambda(\LL):= \inf\left\{\iR \LL(y)\nabla v\cdot\nabla v\,dy: v\in \Csc, \;\iR |\nabla v|^2\,dy=1\right\}.
\]
\par
As long as $\Lambda(\LL)>0$, existence and uniqueness of the solution to \eqref{eq.Dir} is guaranteed by Lax-Milgram's lemma.

\par
Further, according to classical results in the theory of homogenization, under condition \eqref{def.alp-ver-str-ell} the solution $\ue\in H^1_0(\OR)$ of \eqref{eq.Dir} satisfies
$$
\ba{cl}
\ue \wk u, & \mbox{weakly in }H^1_0(\OR)
\\[2mm]
\LL(x/\e)\nabla \ue \wk \LL^0\nabla u, & \mbox{weakly in } L^2(\ORR)\\[2mm]
{-\,\div\big(\LL^0\nabla u\big)=f},&
\ea
$$
with $\LL^0$ given by \eqref{form-hom-cl}. The same result holds true when  \eqref{def.alp-ver-str-ell} is replaced by the condition that $\Lambda(\LL)>0$; see \cite{francfort92}.

When $\Lambda(\LL)=0$, the situation is more intricate. A first result was obtained in \cite[Theorem 3.4(i)]{geymonat.muller.triantafyllidis93}, namely
\begin{theorem}
\label{thm.gmt1}
If $\Lambda(\LL)\ge 0$ and
\[
\LLper(\LL):=\inf\left\{\iY \LL(y)\nabla v\cdot\nabla v\ dy: v\in H^1_{\rm per}(\YR),\ \iY |\nabla v|^2 dy=1 \right\}>0,
\]
then,
${\mathscr I}^\e \wkgh {\mathscr I}^0,$ 
 with $\LL^0$ given by \eqref{form-hom-cl}.
\end{theorem}

This was very recently improved by {\sc A. Braides \& M. Briane} as reported in \cite[Theorem 2.4]{briane.pallares}. The result is as follows:

\begin{theorem}
\label{thm.bpm}
If $\Lambda(\LL)\ge 0$, then, ${\mathscr I}^\e \wkgh {\mathscr I}^0,$ with $\LL^0$ given
\be{eq.hom-for}
\LL^0M\cdot M:= \inf\left\{\iY \LL(y)(M+\nabla v)\cdot (M+\nabla v)\,dy:v\in H^1_{\rm per}(\YR)\right\}.
\ee
\end{theorem}
Note that dropping the restriction that $\LLper(\LL)$ (which is always above $\Lambda(\LL)$) be positive changes the minimum in \eqref{form-hom-cl} into an infimum in \eqref{eq.hom-for}.

\medskip

 As announced in the introduction, we are only interested in the kind of  two-phase mixture that can lead, in the layering case,  to the degeneracy first observed in \cite{gutierrez99}.
 Specifically, we assume  the existence of $2$ isotropic phases $\ZZ_1,\ZZ_2$ of $\T$ -- and of the associated subsets $Z_1$ and $Z_2$ of $\Y$, {or still $Z_1^\#$ and $Z_2^\#$ of $\R^2$ (see notation)} --
 such that 
 \be{hyp-ph}\begin{cases}
 \ZZ_1,\ZZ_2 \mbox{ are open, } C^2 \mbox{ subsets of }\T;
 \\[2mm]
\ZZ_1\cap \ZZ_2=\mbox{\O}\quad\mbox{and}\quad \bar \ZZ_1\cup\bar \ZZ_2=\T;
\\[2mm]
\ds \mbox{ $Z^\#_2$ has an unbounded component in $\R^2$, denoted by $X_2^\#$, and $X_2^\#\cap \Y=Z_2$};
\\[2mm]
\ZZ_1 \mbox{ has a finite number of connected components in }\T.
\end{cases}
\ee
We denote henceforth by $\theta\in (0,1)$ the volume fraction of $\ZZ_1$ in $\T$.
\begin{figure}[ht]
\includegraphics[scale=.25]{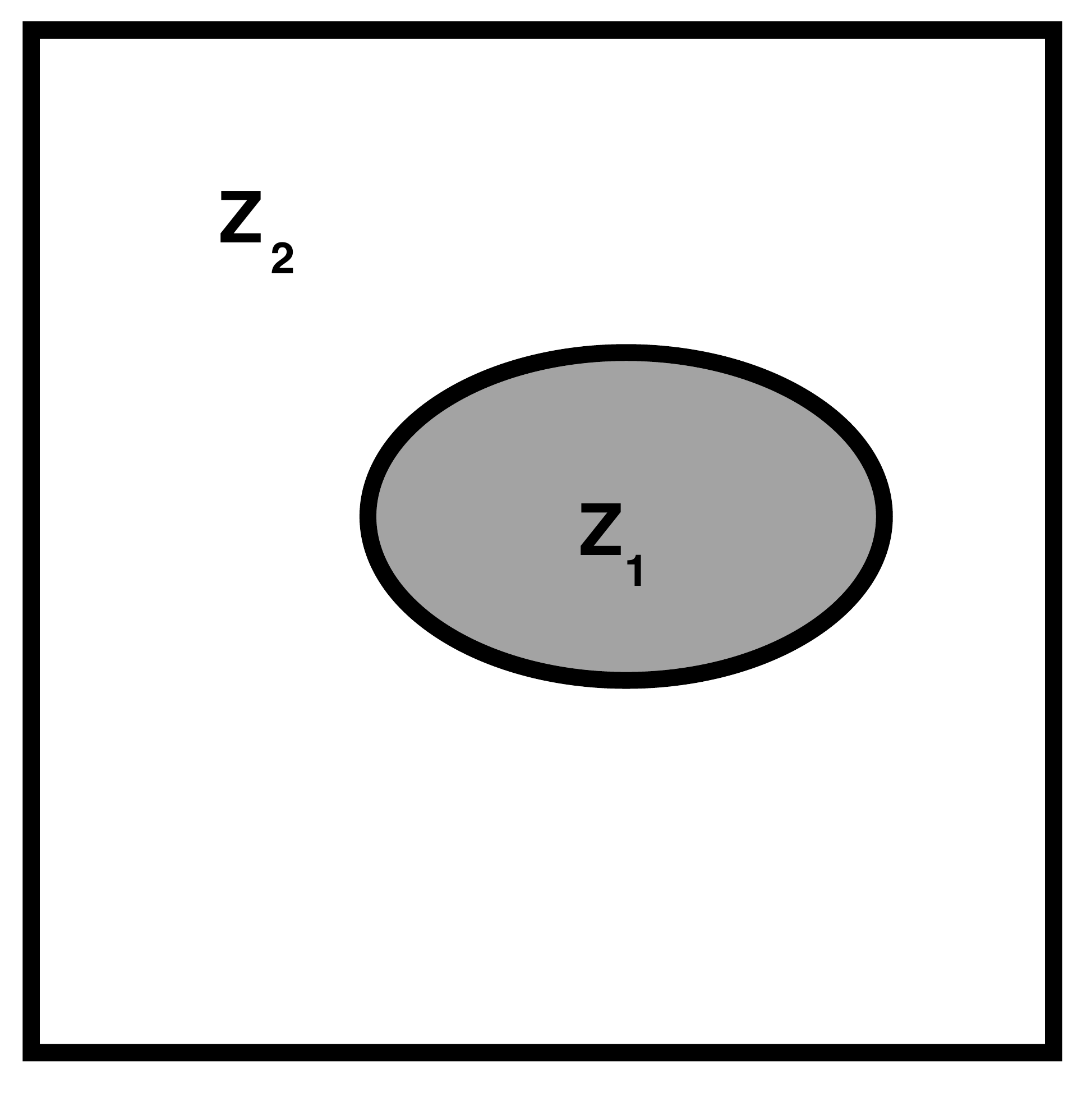}\hfill\includegraphics[scale=.23]{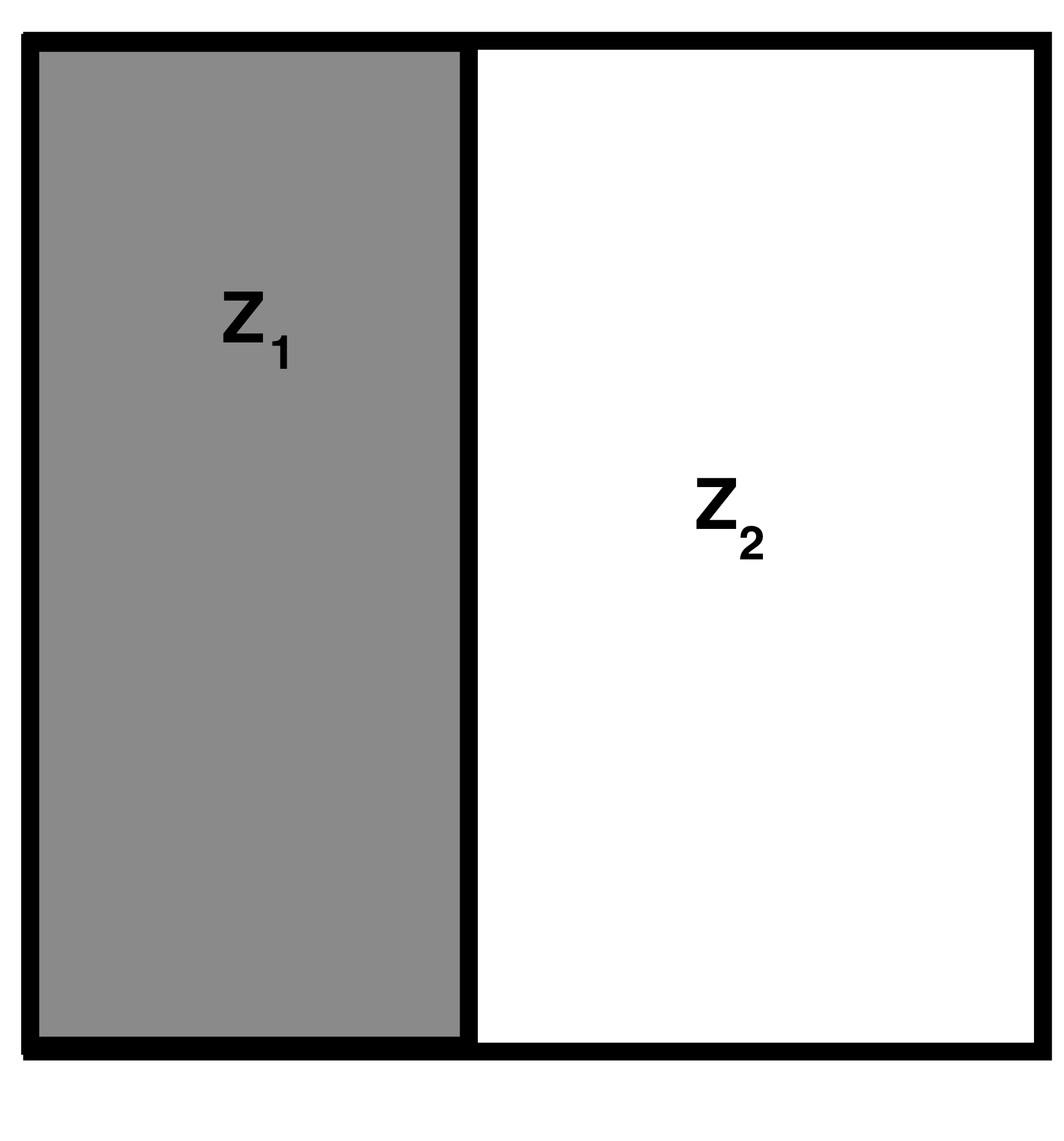}
\caption{\small Typical allowed micro-geometries: inclusion of the good material or layering.}
\label{fig1}
\end{figure}

We then define 
\be{hyp-lam-mu}
\ba{cc}
\LL(y)M= \lambda(y)\,\tr(M)\,I_2+2\mu(y)M, & y\in \T,\ M\in\mtwo
\\[1mm]
\lambda(y)=\lambda_i,\ \mu(y)=\mu_i, & \mbox{in } \ZZ_i,\ i=1,2
\\[1mm]
\ds 0<-\lambda_2-\mu_2= \mu_1<\mu_2,\; \lambda_1+\mu_1>0. &
\ea
\ee
which implies in particular that
$$
\lambda_2+2\mu_2>0,
$$
that is that phase 2 is only  strongly elliptic ($\lambda_2+\mu_2<0$) while phase 1 is very strongly elliptic  ($\lambda_1+\mu_1>0$).
 
Then the following result, which brings together  \cite[Theorem 2.2]{briane.francfort2015} and \cite[Theorem 2.1]{francfort.gloria}, holds true:
\begin{theorem}
\label{thm.lambda}
Under assumptions, \eqref{hyp-ph}, \eqref{hyp-lam-mu},
$
\Lambda(\LL)\ge 0$ and $\LLper(\LL)>0$.
\end{theorem}

Consequently,  Theorem \ref{thm.gmt1} can be applied to the setting at hand and we obtain the following
\begin{corollary}\label{cor}
Set, under assumptions \eqref{hyp-ph}, \eqref{hyp-lam-mu},
$$
{\mathscr J}^\e(v):=\displaystyle\iO \LL(x/\e) \nabla v\cdot\nabla v\,dx
$$
with $\LL^0$ given by \eqref{form-hom-cl} and 
$$
 {\mathscr J}^0(v):=\displaystyle\iO \LL^0 \nabla v\cdot\nabla v\,dx.
$$
Then ${\mathscr J}^\e\wkgh  {\mathscr J}^0$.
\end{corollary}

Our goal in the next section is to prove that the Corollary remains true when adding to ${\mathscr J}^\e$ a zeroth order term of the form 
\[
\int_\Omega |v|^2\,dx
\]
and replacing the weak topology on bounded sets of $H^1_0(\OR)$ by that on bounded sets of $L^2(\OR)$.

\section{The elliptic results}\label{sec.gam}
{Consider
$\LL(y)$  given by \eqref{hyp-lam-mu} and $\LL^0$ given by \eqref{eq.hom-for}.}
Set, for $v\in L^2(\OR)$,
\[
{\mathscr I}^\e(v):=\left\{\begin{array}{cl}
\displaystyle\iO\big\{\LL(x/\e) \nabla v\cdot\nabla v+|v|^2\big\}\,dx, & v\in H^1_0(\OR)
\\[2mm]
\infty, & \mbox{else}.
\end{array}\right.
\]
Also define the following two functionals:
\be{eq.def-iep-iop}
 {\mathscr I}^0(v):=\left\{\begin{array}{cl}
 \displaystyle\iO \big\{\LL^0 \nabla v\cdot\nabla v+|v|^2\big\}\,dx, & v\in H^1_0(\OR)
 \\[2mm]
 \infty, & \mbox{else};
 \end{array}\right.
\ee
and, under the additional assumption that
\[
\LL^0_{2222}=0,
\]
\be{eq.def-iep-iopbis}
 {\mathscr I}^{1/2}(v):=\left\{\begin{array}{cl}
 \displaystyle\iO \big\{\LL^0 \nabla v\cdot\nabla v+|v|^2\big\}\,dx, & \ds v\in \XX
 \\[2mm]
 \infty, &\mbox{else},\end{array}\right.
\ee
where, if $\nu$ is the exterior normal on $\partial\Om$,
\begin{multline}\label{eq:X}
\XX:=\Big\{ v\in L^2(\Om;\R^2): v_1\in H^1_0(\Om),\; v_2\in L^2(\Om),
\\[2mm]
\frac{\partial v_2}{\partial x_1}\in L^2(\Om)\;\;\mbox{and}\;\;v_2\,\nu_1=0 \mbox{ on }\partial\Om\Big\}.
\end{multline}

\begin{remark}\label{cross-terms}
In  \eqref{eq.def-iep-iopbis} the cross terms
 $$\int_\Om\frac{\partial u_1}{\partial x_1}\frac{\partial u_2}{\partial x_2}\ dx$$
 must be replaced  by
 $$\int_\Om\frac{\partial u_1}{\partial x_2}\frac{\partial u_2}{\partial x_1}\ dx$$
 so that, provided that $\LL^0_{2222}=0$, which is the case 
in the specific setting at hand,  
the expression $\int_\Om \LL^0 \nabla u\cdot \nabla u\ dx$ has a meaning for $u\in \XX$ and boils down to the classical one when $u\in H^1_0(\OR)$.
\end{remark}

\begin{remark}\label{structureX}
It is immediately checked that  $\XX$ is a Hilbert space when endowed with the following inner product:
$$
\langle u,v\rangle_\XX:= \int_\Om u\cdot v\ dx+\int_\Om \nabla u_1\cdot\nabla v_1\ dx+\int_\Om \frac{\partial u_2}{\partial x_1}\frac{\partial v_2}{\partial x_1}\ dx.
$$

Furthermore, $C^\infty_c(\Om;\R^2)$ is a dense subspace of $\XX$, provided that $\Om$  is $C^1$.  Indeed, take $u\in \XX$. The first component $u_1$ is in $H^1_0(\OR)$. Defining
\[
{
\check u_2(x):=\begin{cases}u_2(x), & x\in\Om
\\[1mm]
0,& \mbox{else}
\end{cases}
}
\]
we have, thanks to the boundary condition in the definition \eqref{eq:X} of $\XX$,
\[
\int_{\R^2}\check u_2 \frac{\partial \ph}{\partial x_1}\ dx+ \int_{\R^2}\frac{\partial u_2}{\partial x_1}\ph\ dx=0.
\]
for any $\ph\in C^\infty_c(\R^2)$, that is
\[
\check u_2, \frac{\partial \check u_2}{\partial x_1} \in L^2(\R^2)\quad\mbox{with}\quad
{
\frac{\partial \check u_2}{\partial x_1}(x)=
\begin{cases}\ds\frac{\partial  u_2}{\partial x_1}(x), & x\in\Om
\\[3mm]
0, & \mbox{else.}
\end{cases}
}
\]

Because $\Om$ has a $C^1$-boundary, we can always assume, thanks to the implicit function theorem,  that, at each point $x^0\in \partial\Om$, there exists a ball {$B(x^0,r_{x^0})$} and a $C^1$-function $f: \R\to\R$ such that 
\[
\Om\cap B(x^0,r_{x^0})=\{(x_1,x_2)\in B(x^0,r_{x^0}): x_2>f(x_1)\}
\]
or
\[
\Om\cap B(x^0,r_{x^0})=\{(x_1,x_2)\in B(x^0,r_{x^0}): x_1>f(x_2)\}.
\]
In the first case, we translate $\check u$ in the direction $x_2$, thereby setting $\check u_2^t(x_1,x_2):=
\check u_2(x_1,x_2-t), \; t>0$, while, in the second case, we translate $\check u_2$ in the direction $x_1$, thereby setting $\check u_2^t(x_1,x_2):=
\check u_2(x_1-t,x_2),\; t>0$. This has the effect of creating a new function $\check u_2$ which is identically null near $
\Om\cap B(x^0,r_{x^0}).$ We then mollify this function with a mollifier $\ph^t$, with support depending on $t$, thereby creating yet a new function ${\tilde{u}^{t}_2}=\ph^t*\check u_2^t\in C^\infty_c(\Om\cap B(x^0,r_{x^0})$ which will be such that
\[
\lim_t \left\{\|{\tilde{u}^{t}_2}- u\|_{L^2(\Om\cap B(x^0,r_{x^0}))}+\big\|\frac{\partial{\tilde{u}^{t}_2}}{\partial x_1}- \frac{\partial  u_2}{\partial x_1}\big\|_{L^2(\Om\cap B(x^0,r_{x^0}))}\right\}=0.
\]
A partition of unity of the boundary and a diagonalization argument then allow one to construct a sequence of $C^\infty_c(\Om)$-functions such that the same convergences take place over $L^2(\Om)$. \hfill\P
\end{remark}

\medskip

We propose to investigate the (sequential) $\Gamma$-convergence properties of ${\mathscr I}^\e$ to ${\mathscr I}^0$ {or ${\mathscr I}^{1/2}$} for the weak topology on bounded sets of $L^2(\OR)$.

 We will prove the following theorems which  address both the case of a laminate and that of a matrix-inclusion type mixture. The first theorem does not completely characterize the $\Gamma$-limit to the extent that it is assumed {\it a priori} that the target field $u$ lies in $H^1_0(\OR)$. By contrast, the second theorem is a complete characterization of the $\Gamma$-limit but it does restrict the geometry of laminate-like mixtures to be that made of {\it bona fide} layers, {\it i.e.,} straight strips of material.
\begin{theorem}[\sf ``Smooth targets"]
\label{thm.main}
Under assumptions \eqref{hyp-ph}, \eqref{hyp-lam-mu}, there exists a subsequence of $\{\e\}$ (not relabeled) such that
$${\mathscr I}^\e \wkg {\mathscr I},$$
 where, for $u\in H^1_0(\OR)$,  ${\mathscr I}(u)= {\mathscr I}^0(u)$ given by \eqref{eq.def-iep-iop} and $\LL^0$  given by \eqref{eq.hom-for} (and, even better, by \eqref{form-hom-cl}).
\end{theorem}

\begin{theorem}[\sf ``General targets"]
\label{thm.mainbis}
Under assumptions \eqref{hyp-ph}, \eqref{hyp-lam-mu}, then the following holds true:
\begin{itemize}
\item[(i)] If {$\bar{Z}_1\subset\mathring{\Y}$} (the inclusion case)  and if  $\Om$ is a bounded open Lipschitz domain in $\R^2$,  then
$${\mathscr I}^\e \wkg {\mathscr I}^0$$
given by \eqref{eq.def-iep-iop} and $\LL^0$  given by \eqref{eq.hom-for} (and, even better, by \eqref{form-hom-cl});
\item[(ii)] If $Z_1=(0,\theta)\times (0,1)$ (or $(0,1)\times (0,\theta)$) (the straight layer case) and if $\Om$ is a bounded open $C^1$ domain in $\R^2$, then
$${\mathscr I}^\e \wkg \begin{cases}{\mathscr I}^0&\mbox{ if } \theta\ne 1/2\\[2mm]{\mathscr I}^{1/2}&\mbox{ if } \theta= 1/2\;\; \mbox{(the Gutiérrez case)}\end{cases} $$
which are given by \eqref{eq.def-iep-iop}, \eqref{eq.def-iep-iopbis}, respectively,  and  with $\LL^0$  given by \eqref{eq.hom-for} (and, even better, by \eqref{form-hom-cl}).
\end{itemize}
\end{theorem}

\begin{remark}
In strict parallel with Remark 2.6 in \cite{briane.francfort2015}, we do not know whether the result of {those Theorems  still hold} true when $H^1_0(\OR)$ is replaced by $H^1(\OR)$ in the $\Gamma$-convergence statement.
\hfill\P
\end{remark}

\begin{remark}
We could generalize the inclusion condition $(i)$ $\bar{Z}_1\subset\mathring{\Y}$ as follows. Consider $n$ regular compact sets $K_1,\dots, K_n$ of $\R^2$ such that all the translated $K_j+\kappa$ for $j\in\{1,\dots,n\}$ and $\kappa\in\Z^2$, are pairwise disjoint. Then  define the $\Y$-periodic  phase 1 by
\[
Z^\#_1:=\bigcup_{j=1}^n\bigcup_{\kappa\in\Z^2}(K_j+\kappa).
\]
All subsequent results  pertaining to case (i)  extend to this enlarged setting.
\hfill\P
\end{remark}

\begin{remark}\label{sse.l0} The strict strong ellipticity of $\LL^0$ in Theorem \ref{thm.mainbis} is known. 

In case (i), $\LL^0$ remains strictly strongly elliptic. This is  explicitly stated in \cite[Theorem 2.2]{francfort.gloria} under a restriction of isotropy although the proof immediately extends to the fully anisotropic case as well.

In case (ii), strict strong ellipticity is preserved except in the Gutiérrez case ($\theta=1/2$)  in which case $\LL^0_{2222}=0$, as first evidenced in \cite{gutierrez99}.\hfill\P
\end{remark}

Subsections \ref{thmm.main}, \ref{sub.sec.proofbis} are devoted to the proofs of {Theorem \ref{thm.main}, \ref{thm.mainbis}}.
 
\subsection{{Proof of Theorem \ref{thm.main}}}\label{thmm.main}
First,  because of the compactness of the injection mapping from $H^1_0(\OR)$ into $L^2(\OR)$ and in view of Corollary \ref{cor},
$$
{\mathscr I}^\e \wkgh  {\mathscr I}^0
$$
with $\LL^0$ actually given by \eqref{form-hom-cl} (a min in lieu of an inf).
We want to prove that the same result holds for the weak $L^2$-topology, {at least for a subsequence of $\{\e\}$. By a classical compactness result we can assert the existence of a subsequence of $\{\e\}$ such that the $\Gamma\mbox{-}\lim$ exists. Our goal is to show that that limit, denoted by $\mathscr I(u)$, is precisely $\mathscr I^0(u)$ when $u\in H^1_0(\OR)$.}
Clearly, the $\Gamma\mbox{-}\limsup$ inequality will {\em a fortiori} hold in that topology, {\it provided that the target field $u\in H^1_0(\OR)$}. It thus remains to address the proof of the $\Gamma\mbox{-}\liminf$ inequality which is what the rest of this subsection is about.

\medskip

To that end and in the spirit of \cite{gutierrez99}, we add an integrated null Lagrangian to the energy so as to render the energy density pointwise nonnegative. Thus we set, for any $M\in \mtwo$,
{
\be{Kj}
\KK_j M := \mathbb L_j M + 2\mu_1 {\rm cof}\left(M\right)  = \lambda_j \tr(M) I_2 + \mu_j (M+M^T) + 2\mu_1 {\rm cof}\left(M\right),\;\; j=1,2
\ee
}
(thereby taking $c$ at the end of the introduction to be {$4\mu_1$}) so that
\[
\KK_j M\cdot M = \mathbb L_j M\cdot M + 4\mu_1 \det\left(M\right)\geq 0,\quad j=1,2
\]
and define
$$
\KK(y)\equiv \KK_j \mbox{ in }\ZZ_j, \quad j=1,2.
$$
Because the determinant is a null Lagrangian, for $v\in H^1_0(\OR)$,
\be{form.with.K}
{\mathscr I}^\e(v)= \iO \big\{\KK(x/\e) \nabla v\cdot\nabla v+|v|^2\big\}\,dx
\ee
Consider a sequence  $\{u^\e\}_\e$ converging weakly in $L^2(\OR)$ to $u\in L^2(\OR)$. Then, for a subsequence (still indexed by $\e$), we are at liberty to assume that
$
\liminf {\mathscr I}^\e(u^\e)
$ is actually a limit. The $\Gamma\mbox{-}\liminf$ inequality is trivial if that limit is $\infty$ so that we can also assume henceforth that, for some $\infty>C>0$,
\be{eq:bd-funct}
{\mathscr I}^\e(u^\e)\le C.
\ee
Further, according to {\em e.g.} \cite[Theorem 1.2]{allaire92}, a subsequence (still indexed by $\e$) of that sequence two-scale converges to some $u^0(x,y)\in L^2(\OYR)$. In other words,
\be{conv2-u}
u^\e \twosc u^0.
\ee
Also, in view of \eqref{eq:bd-funct} and because
 \be{K-nonneg}
 \KK(y) \mbox{  is a nonnegative as a quadratic form}
 \ee
 while clearly all its components are bounded, for yet another subsequence (not relabeled),
$$
\KK(x/\e)\nabla u^\e \twosc H(x,y)\quad\mbox{with } H\in L^2(\OYRR),
$$
and also, for future use,
\be{conv2-sr}
\KK^{\frac12}(x/\e)\nabla u^\e \twosc S(x,y)\quad\mbox{with } S\in L^2(\OYRR).
\ee
In particular,
\be{conv2-efl}
\e\,\KK(x/\e)\nabla u^\e \twosc 0.
\ee
Take $\Phi(x,y) \in \Cinf$  with compact support in $\Om\times \ZZ_1$. From \eqref{conv2-efl} we get, with obvious notation, that
\begin{multline*}
0=-\lim_{\e\to 0}\int_{\Om}\e\,\KK(x/\e)\nabla u^\e\cdot\Phi(x,x/\e)\,dx=
\\[2mm]
\lim_{\e\to 0}\;\sum_{ijkh}\int_{\Om}(\KK)_{ijkh}(x/\e)\,u^\ep_k\,\frac{\partial \Phi_{ij}}{\partial y_h}(x,x/\e)\,dx=
\\[2mm]
\sum_{ijkh}\itwo{1} (\KK_1)_{ijkh}\,u^0_k(x,y)\,\frac{\partial \Phi_{ij}}{\partial y_h}\,dx\,dy
=-\itwo{1} \big(\KK_1 \nabla_y u^0(x,y)\big)\cdot  \Phi(x,y)\,dx\,dy,
\end{multline*}
so that
\be{eq:o-grad}
\KK_1 \nabla_y u^0(x,y)\equiv 0\quad\mbox{in } \Om\times {Z^\#_1},
\ee
and similarly
\be{eq:o-gradbis}
\KK_2 \nabla_y u^0(x,y)\equiv 0\quad\mbox{in } \Om\times {Z^\#_2}.
\ee
In view of the explicit expressions {\eqref{Kj} for $\KK_j$,} \eqref{eq:o-grad}, \eqref{eq:o-gradbis} imply that
{
\begin{align*}
\lambda_j\left(\frac{\partial u^0_1}{\partial y_1}+\frac{\partial u^0_2}{\partial y_2}\right)+2\mu_j \frac{\partial u^0_1}{\partial y_1}+2\mu_1 \frac{\partial u^0_2}{\partial y_2}=0
\\[2mm]
\lambda_j\left(\frac{\partial u^0_1}{\partial y_1}+\frac{\partial u^0_2}{\partial y_2}\right)+2\mu_j \frac{\partial u^0_2}{\partial y_2}+2\mu_1 \frac{\partial u^0_1}{\partial y_1}=0
\\[2mm]
\mu_j\left(\frac{\partial u^0_1}{\partial y_2}+\frac{\partial u^0_2}{\partial y_1}\right)-2\mu_1\frac{\partial u^0_1}{\partial y_2}=0
\\[2mm]
\mu_j\left(\frac{\partial u^0_1}{\partial y_2}+\frac{\partial u^0_2}{\partial y_1}\right)-2\mu_1\frac{\partial u^0_2}{\partial y_1}=0.
\end{align*}
}
So, {in phase 1, that is on $Z^\#_1$,} using~\eqref{hyp-lam-mu} we get
\be{eq:ph1}
\frac{\partial u^0_1}{\partial y_1}+\frac{\partial u^0_2}{\partial y_2}=0,\quad 
\frac{\partial u^0_2}{\partial y_1}-\frac{\partial u^0_1}{\partial y_2}=0
\ee
while {in phase 2, that is on $Z^\#_2$,} still using \eqref{hyp-lam-mu} we get
\be{eq:ph2}
\frac{\partial u^0_1}{\partial y_1}=\frac{\partial u^0_2}{\partial y_2},\quad 
\frac{\partial u^0_2}{\partial y_1}=\frac{\partial u^0_1}{\partial y_2}=0.
\ee

From  \eqref{eq:ph1} we conclude that, {in phase 1,}
\begin{equation}\label{eq:lap-ph1}
\triangle_y u^0_1=\triangle_y u^0_2=0.
\end{equation}

\medskip
\noindent {\sf Step 1 -- $u^0$ does not oscillate.}  We now exploit the two previous set of relations under the micro-geometric assumptions of Theorem \ref{thm.main} to demonstrate that
\be{uo-indt-y}
u^0(x,y)=u(x) \mbox{ is independent of }y,
\ee
where, thanks to \eqref{conv2-u},
\be{conv-1sc}
u^\e \wk u, \mbox{ weakly in } L^2(\OR).
\ee

We first notice that, in view of \eqref{eq:ph2} and  because  $X_2^\#$ is connected (see \eqref{hyp-ph}),
\[
u^0_1(x,y)=\alpha(x)\,y_1+\beta(x),\quad u^0_2(x,y)=\alpha(x)\,y_2+\gamma(x),\quad y \in X_2^\#,
\]
for some functions $\alpha(x), \beta(x),\gamma(x)$. By $Y_2$-periodicity of $u^0_i$ and because $X_2^\#$ is unbounded, $\alpha(x)=0$.
Thus $\nabla_y u^0=0$, or equivalently,
\be{eq:uo-ph2}
u^0(x,y)=u(x),\ y\in Z^\#_2\left(=\bigcup_{\kappa\in\Z^2}(X_2^\#+\kappa)\right),
\ee 
for some $u\in L^2(\OR)$.

Consider $\Phi\in C^1_\sharp(\Y;\R^{2\times 2})$ with
\be{nc-Phi}
\sum_{ijh}  (\KK_1-\KK_2)_{ijkh}\Phi_{ij}(y)\nu_h(y)=0\quad\mbox{on } \partial Z^\#_1,
\ee
that condition being necessary for {${\rm div}_y (\KK(y)\Phi(y))$ } to be an admissible test function for two-scale convergence. In \eqref{nc-Phi} $\nu(y)$ denotes the exterior normal to $Z^\#_2$ at~$y$. 

In view of \eqref{conv2-efl}, \eqref{eq:uo-ph2}, \eqref{nc-Phi}, we get that, for any  $\varphi\in C^\infty_c(\Omega; C^\infty_\sharp(\Y)),$   

\begin{multline*}
0=-\lim_{\e\to 0}\int_{\Om}\e\,\KK(x/\ep)\nabla u_\e\cdot\ph(x,x/\e)\Phi(x/\e)\,dx
\\
=\lim_{\e\to 0}\int_{\Om}\left[\frac\partial{\partial y_h}\big\{(\KK(y))_{ijkh} \varphi(x,y)\Phi_{ij}(y)\big\}\right]\!(x,x/\e)\,(u_\ep)_k\,dx=
\\
\sum_{ijkh}\int_{\Om \times Z_1} \frac\partial{\partial y_h}\big\{(\KK_1)_{ijkh} \varphi(x,y)\Phi_{ij}(y)\big\} u^0_k(x,y)\,dx\,dy \ +
\\
\sum_{ijkh}\int_{\Om \times Z_2} \frac\partial{\partial y_h}\big\{(\KK_2)_{ijkh} \varphi(x,y)\Phi_{ij}(y)\big\} u_k(x)\,dx\,dy=
\\
\sum_{ijkh}\int_{\Om \times Z_1} \frac\partial{\partial y_h}\big\{(\KK_1)_{ijkh} \varphi(x,y)\Phi_{ij}(y)\big\} u^0_k(x,y)\,dx\,dy \ +
\\
\sum_{ijkh}\int_{\Om\times {\partial Z_1}} (\KK_1)_{ijkh} u_k(x)\nu_h(y) \varphi(x,y)\Phi_{ij}(y)\,dx\,d\h_y.
\end{multline*}

Set $v^0(x,y):= u^0(x,y)-u(x)$.
Then,
\[
\int_{\Om \times Z_1} {\rm div}_y\big\{\varphi(x,y)\,\KK_1\Phi(y)\big\}\cdot v^0(x,y)\,dx\,dy=0.
\]
Now take $\ph\in C^\infty_c(\Om\times\R^2)$. Using the periodized function 
\[
\ph^\#(x,y):=\sum_{\kappa\in\Z^2}\ph(x,y+\kappa)
\]
as new test function we obtain
\be{eq:v0}
\begin{array}{ll}
\ds 0 & \ds =\int_{\Om \times Z_1} {\rm div}_y\big\{\varphi^\#(x,y)\,\KK_1\Phi(y)\big\}\cdot v^0(x,y)\,dx\,dy
\\[5mm]
& \ds =\sum_{\kappa\in\Z^2}\int_{\Om \times (Z_1+\kappa)} {\rm div}_y\big\{\varphi(x,y)\,\KK_1\Phi(y)\big\}\cdot v^0(x,y)\,dx\,dy
\\[5mm]
& \ds =\int_{\Om \times Z_1^\#} {\rm div}_y\big\{\varphi(x,y)\,\KK_1\Phi(y)\big\}\cdot v^0(x,y)\,dx\,dy.
\end{array}
\ee
Now simple algebra using the explicit expression for $\KK_1,\KK_2$ as well as \eqref{hyp-lam-mu} shows that, for any $\xi\in\R^2$ and $\nu \in S^1$, there exists a unique matrix $\Phi$ such that
\be{bc-Phi}
\ds (\KK_1\Phi)\,\nu=(\KK_2\Phi)\,\nu=\xi,
\ee
so that, in particular, \eqref{nc-Phi} can always be met, provided that each connected component of $\ZZ_1$ has a $C^2$ boundary because the normal $\nu(y)$ is then a $C^1$-function of $y\in \partial \ZZ_1$ so that one can define $\Phi(y)$ satisfying \eqref{bc-Phi} as a $C^1$ function on $\partial \ZZ_1$, hence by {\em e.g.} Whitney's extension theorem as a $C^1$ function on $\T$.

Consider a connected component ${Z^\#}$ of $Z^\#_1$ in $\R^2$. 
Recall that $\KK_1\nabla_yv^0=0$ in $Z^\#_1$.  In view of \eqref{eq:v0}, \eqref{bc-Phi}, and the arbitrariness of $\ph$, $\xi$,  an integration by parts yields that $v^0(x,\cdot)$ has a trace on $\partial Z^\#$ which satisfies
\begin{equation}\label{eq:trace}
v^0(x,\cdot)=0\mbox{ on }\partial Z^\#.
\end{equation}

Fix $x$. According to \eqref{eq:ph1},  there exists a potential $\zeta_x\in H^1(Z^\#\cap (-R,R)^2)$ for any $R>0$, such that
\[
v^0(x,y)=R^\perp \nabla \zeta_x(y)
\]
and
\[
\triangle_y \zeta_x=0\;\;\mbox{in } Z^\#.
\]
Further, in view of \eqref{eq:trace}, 
\[
R^\perp\nabla\zeta_x\cdot\nu=\nabla\zeta_x\cdot\nu^\perp=0\;\;\mbox{on } \partial Z^\#,
\]
so that $\zeta_x$ is constant on each connected component of $\partial Z^\#$.
Thus, by elliptic regularity $\zeta_x\in H^2(Z^\#\cap (-R,R)^2)$ for any $R>0$, hence $v^0\in H^1(Z^\#\cap (-R,R)^2)$ for any $R>0$.
Thanks to \eqref{eq:lap-ph1}, \eqref{eq:trace} and the periodicity of $v^0(x,\cdot)$, we conclude that $v^0\equiv0$, hence \eqref{uo-indt-y}.
}

\bigskip
\noindent {\sf Step 2 -- Identification of the $\Gamma\mbox{-}\liminf$.}  
Consider  $\Phi \in L^2_\sharp(\Y;\mn)$ such that 
\be{eq:Zxy}
\div\big(\KK^{\frac12}(y)\Phi(y)\big)=0\quad\mbox{in }\R^2,
\ee
or equivalently,
\[
\int_{\Y}\KK^{\frac12}(y)\Phi(y)\cdot\nabla\psi(y)\,dy=0\quad\forall\,\psi\in H^1_\sharp(\Y;\R^2),
\]
 and also consider $\varphi\in C^\infty(\Omb)$.

Then, since $u^\e\in H^1_0(\OR)$ and
in view of \eqref{eq:Zxy},
\begin{multline*}
\iO\varphi(x)\,\KK^{\frac12}(x/\e)\nabla u^\e \cdot \Phi(x/\e)\,dx=
-\sum_{ijkh}\iO u_k^\e\,\cdot\,\KK^{\frac12}_{ijkh}(x/\e)\Phi_{ij}(x/\e)\frac{\partial \varphi}{\partial x_h}(x)\,dx.
\end{multline*}
Recalling
\eqref{conv2-u}, \eqref{conv2-sr}, we can pass to the two-scale limit in the previous expression and obtain, thanks to \eqref{uo-indt-y},
\be{eq:ident-s}
\displaystyle \itwot\varphi(x) S(x,y) \cdot \Phi(y)\,dx\,dy=
\displaystyle -\sum_{ijkh}\itwot u_k(x) \cdot \KK^{\frac12}_{ijkh}(y)\Phi_{ij}(y)\frac{\partial \varphi}{\partial x_h}(x)\,dx\,dy.
\ee

Assume henceforth that $u\in H^1(\Omega;\R^2)$. Then,  \eqref{eq:ident-s} implies that
$$
\itwot S(x,y) \cdot \Phi(y)\,\varphi(x) \,dx\,dy= \itwot \KK^{\frac12}(y)\nabla_x u(x) \cdot \Phi(y)\,\varphi(x) \,dx\,dy.
$$

By density, the result still holds with the test functions $\varphi(x)\Phi(y)$ replaced by the set of 
$\Psi(x,y)\in L^2(\Om;L^2_\sharp(\Y;\mn))$ 
such that
\[
{\rm div}_y\big(\KK^{\frac12}(y) \Psi(x,y)\big)=0\quad\mbox{in }\R^2,
\]
or equivalently, due to the symmetry of $\KK(y)$,
\[
\int_{\Om\times\Y}\Psi(x,y)\cdot\KK^{\frac12}(y)\nabla_yv(x,y)\,dy=0\quad\forall\,v\in L^2(\Om;H^1_\sharp(\Y;\R^2)).
\]
The $L^2(\Om;L^2_\sharp(\Y;\mn))$-orthogonal to that set is  the $L^2$-closure of
\[
\Ks_\nabla:=\left\{\KK^{\frac12}(y)\nabla_yv(x,y): v\in L^2(\Om;H^1_\sharp(\Y;\R^2))\right\}.
\]

Thus, 
$$
S(x,y)=\KK^{\frac12}(y)\nabla_x u(x)+\xi(x,y)
$$
for some $\xi$ in the closure of $\Ks_\nabla$ and there exists a sequence
\[
v_n\in L^2(\Om;H^1_\sharp(\Y;\R^2))
\]
such that $\KK^{\frac12}(y)\nabla_yv_n \to \xi$, strongly in $L^2(\Om;L^2_\sharp(\Y;\mn))$.

We now appeal to \cite[Proposition 1.6]{allaire92} which yields
\begin{multline}\label{2sc-lsc}
\liminf_{\e\to 0}\,\|\KK^{\frac12}(x/\e)\nabla u^\e\|^2_{L^2(\Om;\mn))}\ge \|S\|^2_{L^2(\OYRR)}=\\ \lim_n
\|\KK^{\frac12}(y)\nabla_x u(x)+\KK^{\frac12}(y)\nabla_yv_n\|^2_{L^2(\Omega\times\Y;\R^2)}.
\end{multline}
But recall that
$$
\|\KK^{\frac12}(x/\e)\nabla u^\e\|^2_{L^2(\Om;\mn))}=\iO \KK(x/\e)\nabla u^\e\cdot\nabla u^\e\,dx= \iO \LL(x/\e)\nabla u^\e\cdot\nabla u^\e\,dx
$$
because the determinant is a null Lagrangian.

Thus, from \eqref{2sc-lsc} and by weak $L^2$-lower semi-continuity of $\|u^\e\|_{L^2(\Om;\R^2)}$ we conclude that
\begin{multline}\label{eq:ga-liminf}
{\liminf_{\e\to 0}}\,{\mathscr I}^\e(u^\e)\ge
\\
\lim_n\itwot \KK(y)(\nabla_x u(x)+\nabla_yv_n(x,y))\cdot(\nabla_x u(x)+\nabla_yv_n(x,y))\,dx\,dy +\iO \!|u|^2\,dx\ge
\\
\inf\left\{ \itwot\!\!\!\KK(y)(\nabla_x u(x)+\nabla_yv(y))\cdot(\nabla_x u(x)+\nabla_yv(y))\,dx\,dy: v\in H^1_\sharp(\Y;\R^2)\right\}
\\
+\iO \!|u|^2\,dx.
\end{multline}
In the light of the definition \eqref{eq.hom-for} for $\LL^0$, we {finally get}
\[
{\liminf_{\e\to 0}}\,{\mathscr I}^\e(u^\e)\ge \iO \left\{\LL^0 \nabla_x u\cdot \nabla_x u+|u|^2\right\}\,dx
\]
provided that $u\in H^1(\Omega;\R^2)$, {hence, {\it a fortiori} provided that $u\in H^1_0(\Omega;\R^2)$}.

The proof of Theorem \ref{thm.main} is complete.

\bigskip

\subsection{{Proof of Theorem \ref{thm.mainbis}}}\label{sub.sec.proofbis}
Recall that, in the proof of Theorem \ref{thm.main}, we were at liberty to assume that
\[
{\mathscr I}^\e(u^\e)\le C<\infty,
\]
otherwise the $\Gamma$-$\liminf$ inequality is trivially verified. Consequently, if we can show that, under that condition, the target function $u$ is in $H^1_0(\OR)$, then we will be done as remarked at the onset of Subsection \ref{thmm.main}. Such will be the case except when dealing with straight layers (case (ii)) under the condition that $\theta=1/2$.
In that case we will have to show that, for those target fields $u$ that are not in $H^1_0(\OR)$, a recovery sequence for the $\Gamma\mbox{-}\limsup$ (in)equality can be obtained by density.

Returning to \eqref{eq:ident-s},  setting
\be{eq:rgeK}
N_{kh}:= \sum_{ij}  \int_{\Y} \KK^{\frac12}_{ijkh}(y)\Phi_{ij}(y)\,dy
\ee
and varying $\varphi$ in $C^\infty_c(\Omega)$, we conclude that
\be{eq:mnablau}
N\cdot \nabla u\in L^2(\Omega).
\ee

We now remark that $\KK(y)$, a symmetric mapping on  $\R^{2\times 2}$, has for  eigenvalues $2(\lambda(y)+\mu(y)+\mu_1)$, {$2\mu_1$} and $2(\mu(y)-\mu_1)$ with, if $\mu_2\neq 2\mu_1$, eigenspaces respectively generated by 
\par
\centerline{$I_2,\quad R^\perp$,}
\par\noindent
and, for the last eigenvalue, by
\par
\centerline{$G:=\left(\begin{smallmatrix}1&\ 0\\0&-1\end{smallmatrix}\right),\quad H:=\left(\begin{smallmatrix}0&1\\1&0\end{smallmatrix}\right)$.}
\par\noindent
Consequently, its kernel for $y\in Z_2$ is 
\be{eq:KerK2}
{\rm Ker}\,\big(\KK(y)\big)={\rm Ker}\,(\KK_2):=\left\{\gamma I_2,\; \gamma \mbox{ arbitrary in }\R\right\},
\ee
while its kernel for $y\in Z_1$ is
\be{eq:KerK1}
{\rm Ker}\,\big(\KK(y)\big)={\rm Ker}\,(\KK_1):=\left\{\left(\begin{smallmatrix}\alpha&\ \beta\\\beta&-\alpha\end{smallmatrix}\right):\alpha,\beta \mbox{ arbitrary in }\R\right\}.
\ee

\medskip

\noindent {\sf Step 1 -- Case (i).} First assume that $\bar{Z}_1\subset\mathring{\Y}$ and that $M\in \MS$.  We  then define 
\[
\Psi(y):=\frac12 \left\{My\cdot y-\ph(y-\kappa)M(y-\kappa)\cdot (y-\kappa)\right\}\quad\mbox{for any }y\in\Y+\kappa,\ \kappa\in\Z^2,
\]
with {$\ph\in C^2_c(\mathring{\Y})$,} $\ph\equiv 1$ in $Z_1$.
Then clearly, $\nabla\Psi-My \in H^1_{\sharp}(\Y;\R^2)$. Further, $\nabla \Psi=M\kappa$ in $Z_1+\kappa$ hence $\nabla^2\Psi\equiv0$ in $Z_1+\kappa$ while $\nabla^2\Psi R^\perp\in {\rm Ker}^\perp\,(\KK_2)$ thus belongs to the range of $\KK_2$ in $Z_2+\kappa$. 

It is thus meaningful to define $\Phi(y):=\KK^{-\frac12}(y)\big(\nabla^2 \Psi(y)\,R^\perp\big)$ where $\KK^{-\frac12}$ is the pseudo-inverse of~$\KK^{\frac12}$ (see the notation at the close of the introduction). 

We get
\[
\int_{\Y} \KK^{\frac12}(y)\Phi(y)\,dy=\int_{\Y} \nabla^2 \Psi(y)\,R^\perp\,dy= MR^\perp,
\]
while $\Phi$  satisfies \eqref{eq:Zxy} since
for any $v\in H^1_{\rm loc}(\R;\R^2)$ with periodic gradient,
\[
\div(\nabla v\,R^\perp)=0\quad\mbox{in }\R^2,
\]
or equivalently,
\[
\int_{\Y}\nabla v(y)\,R^\perp\cdot\nabla\psi(y)\,dy=0\quad\forall\,\psi\in H^1_\sharp(\Y;\R^2).
\]
We finally obtain {by \eqref{eq:mnablau}} that
$MR^\perp\cdot\nabla u\in L^2(\Omega)$. Since, when $M$ spans $\MS$, $N:=MR^\perp$ spans the set of all $2\times 2$ trace-free matrices , we infer from \eqref{eq:mnablau} that
\[
\ds\frac{\partial u_1}{\partial x_2},\; \frac{\partial u_2}{\partial x_1},\;\frac{\partial u_1}{\partial x_1}-\frac{\partial u_2}{\partial x_2} \mbox{ are in } L^2(\Omega).
\]
This is equivalent to stating that $\esym(R^\perp u)\in L^2(\Omega;\MS)$. Since $\Omega$ is Lipschitz, Korn's inequality allows us to conclude that $R^\perp u\in H^1(\OR)$, hence that 
\be{uinH1}
u\in H^1(\OR)
\ee 
in that case.

Since, for an arbitrary trace-free matrix  $N$, we can choose $\Phi$ constrained by \eqref{eq:Zxy} so that \eqref{eq:rgeK} is satisfied, then actually  
\be{eq:uninH10}
u\in H^1_0(\Omega;\R^2).
\ee
Indeed, take $x^0\in \partial\Omega$ to be a Lebesgue point for $u\lfloor\partial\Omega$ -- which lies in particular in $L^2_{\h}(\partial\Omega;\R^2)$ -- as well as for $\nu(x^0)$, the exterior normal to $\Omega$ at $x^0$. Then take an arbitrary trace-free $N$ and the associated $\Phi$.

By \eqref{uinH1} we already know that $u\in H^1(\Omega;\R^2)$, so that \eqref{eq:ident-s} reads as
\begin{multline*}
\itwot\varphi(x) S(x,y) \cdot \Phi(y)\,dx\,dy=
\\
\int_\Omega N\cdot \nabla u(x)\,\varphi(x)\,dx - \!\int_{\partial\Omega} N\nu(x)\cdot u(x)\,\varphi(x)\,d\h.
\end{multline*}
But,  taking first $\varphi \in C^\infty_c(\Omega)$ and remarking that, in such a case, the first two integrals are equal and bounded by a constant times $\|\ph\|_{L^2(\Om)}$, we immediately conclude that, for any $\varphi\in C^\infty(\Omb)$,
$$
\int_{\partial\Omega} N\nu(x)\cdot u(x)\,\varphi(x)\,d\h=0.
$$
Thus, $N\nu(x)\cdot u(x)=0$  $\h$-a.e. on $\partial \Omega$, hence, since $x^0$ is a Lebesgue point, $N\nu(x^0)\cdot u(x^0)=0$ from which it is immediately concluded that  $u_i(x^0)=0,\; i=1,2$, hence \eqref{eq:uninH10}.

But, in such a case we can apply Theorem \ref{thm.main} which thus delivers the $\Gamma$-limit.

\medskip

\noindent {\sf Step 2 -- Case (ii).} Assume now that $Z_1$ is a straight layer, that is that there exists $0<\theta<1$ such that 
$(0,\theta)\times (0,1)= Z_1\cap\mathring \Y$. 

For an arbitrary matrix $M\in \R^{2\times 2}$ define $v(y)$ as
\[
v(y) := My + \left(\int_0^{y_1} [\chi(t)-\theta]\,dt\right) \xi, \; \xi\in \R^2,
\]
where $\chi(y)$ is the characteristic function of phase 1 with volume fraction
$$
\iYt \chi(y)\ dy:=\theta.
$$ 
Then,   $v(y) - My$ is $Y_2$-periodic and
\[
\nabla v(y) = \chi(y_1) (M+(1-\theta)\xi\otimes e_1) + (1-\chi(y_1)) (M-\theta\xi\otimes e_1).
\]
According to \eqref{eq:KerK1}, \eqref{eq:KerK2}, for $\nabla v(y)R^\perp$ to be in the {range} of $\KK(y)$ we must have both
\[
(M+ (1-\theta)\xi\otimes e_1)R^\perp \cdot G = (M+(1-\theta)\xi\otimes e_1)R^\perp \cdot H = 0
\]
{\it i.e.,} $(1-\theta)\xi_1 = M_{22} - M_{11}$, $(1-\theta)\xi_2 = -M_{12} - M_{21}$
and
\[
(M-\theta\xi\otimes e_1)R^\perp \cdot I_2 = 0,
\]
{\it i.e.,} $\theta\xi_2 = -M_{12} + M_{21}$.

Since  $\xi$ can be  arbitrary, this imposes as sole condition on $M$ that 
$$(2\theta-1)M_{12}+M_{21} = 0.$$
Then,
 \[
 \iYt\nabla v(y) R^\perp\ dy= MR^\perp= \begin{pmatrix} M_{12} & -M_{11} \\ M_{22} & (2\theta-1)\,M_{12}\end{pmatrix} = N_\theta
 \]
 where 
\[
 N_\theta=\left(\begin{array}{cc} a&\; c\\[2mm]b& (2\theta -1)a\end{array}\right), \;\mbox{ with }a,b,c \mbox{ arbitrary.}
\]

In view of \eqref{eq:mnablau}, we obtain that
\[
\frac{\partial u_1}{\partial x_2},\ \frac{\partial u_2}{\partial x_1},\ \frac{\partial u_1}{\partial x_1}+{(2\theta-1)}\,\frac{\partial u_2}{\partial x_2}\mbox{ are in }L^2(\Om),
\]
or equivalently when $\theta\ne1/2$,
\[
{\mathbf E}(P_\theta u)\in L^2(\Om;\R^{2\times 2})\quad\mbox{with}\quad P_\theta:=\begin{pmatrix} 0 & {2\theta-1} \\ 1 & 0 \end{pmatrix}.
\]
Using Korn's inequality once again, we thus conclude that $u\in H^1(\Om;\R^2)$, except  when $\theta=1/2$ in which case $\ds\frac{\partial u_2}{\partial x_2}$ might not be in $L^2(\Om)$.

 \begin{remark}\label{gen.char.m22=0}
 Actually, when $\theta=1/2$, then all $\Phi$'s that are such that \eqref{eq:Zxy} is satisfied produce, through \eqref{eq:rgeK}, a matrix $M$ with $M_{21}=0$, hence a matrix $N:=MR^\perp$ such that $N_{22}=0$. 
 
 Indeed, the existence of $\Phi$ is equivalent to that of $v(y)=My+w(y)$ with $w\in H^1_\sharp(\Y;\R^2)$ such that
 $$
 \nabla v(y) R^\perp= \KK^{\frac 12}(y_1)\Phi(y)
 $$
 which implies that, for a.e. $y_1\in (0,1)$,
 $$
 \KK^{\frac 12}(y_1)\int_0^1\Phi(y_1,t)\ dt=\int_0^1\nabla v(y_1,t)R^\perp \ dt.
 $$
 In view of \eqref{eq:KerK1}, \eqref{eq:KerK2}, the last relation yield in particular that 
 $$
 \begin{cases}
 \ds v_1(y_1,1)-v_1(y_1,0)=- \int_0^1\frac{\partial v_2}{\partial y_1}(y_1,y_2)\ dy_2,\; 0\le y_1\le 1/2\\[3mm]
 \ds v_1(y_1,1)-v_1(y_1,0)=+ \int_0^1\frac{\partial v_2}{\partial y_1}(y_1,y_2)\ dy_2,\; 1/2\le y_1\le 1,
 \end{cases}
 $$
 or still, since $v(y)=My+w(y)$ with $w$ $\Y$-periodic,
 $$
 \begin{cases}
 \ds M_{12}=- \int_0^1\frac{\partial v_2}{\partial y_1}(y_1,y_2)\ dy_2,\; 0\le y_1\le 1/2\\[3mm]
 \ds M_{12}=+ \int_0^1\frac{\partial v_2}{\partial y_1}(y_1,y_2)\ dy_2,\; 1/2\le y_1\le 1,
 \end{cases}
 $$
 But then
 $$
 M_{21}=\int_\Y \frac{\partial v_2}{\partial y_1}(y_1,y_2)\ dy_1dy_2=0.
 $$
\vskip-1cm\hfill\P \end{remark}

\vskip1cm
 
 Then, through an argument identical to that used in case (i), we find that, for $x^0$  Lebesgue point for $u_1{\lfloor \partial\Om}$ (and for $u_2{\lfloor \partial\Om}$ as well if $\theta\ne 1/2$), $u_1(x^0)=0$ (and $u_2(x^0)=0$ if $\theta\ne 1/2$) while, if $\theta=1/2$, $u_2\,\nu_1$, which is well defined as an element of $H^{-\frac12}(\Om)$, satisfies $u_2\,\nu_1=0$.

 \medskip
 So, here again, we can apply Theorem \ref{thm.main} provided that $\theta\ne 1/2$. It thus remains to compute the $\Gamma$-limit in case (ii) when $\theta=1/2$.  This is the object of the last step below.
 
 \medskip
 \noindent {\sf Step 3 -- Identification of the $\Gamma$-limit -- case (ii) -- $\theta=1/2$; the Gutiérrez case.}
As far as the $\Gamma\mbox{-}\limsup$ inequality is concerned there is nothing to prove once again, because, as already stated at the onset of Subsection \ref{thmm.main} we know the existence of a recovery sequence for any target field $u\in H^1_0(\OR)$. But, according to Remark \ref{structureX}, $H^1_0(\OR)$ is {\em a fortiori} dense in $\XX$. So any element $u\in \XX$ can be in turn viewed as the limit in the topology induced by the inner product $\langle,\rangle_\XX$ of a sequence $u^p\in H^1_0(\OR)$. Since, as noted in Remark \ref{cross-terms},
${\partial u^p_2}/{\partial x_2}$ does not enter the expression 
\[
\int_\Om \LL^0 \nabla u^p\cdot\nabla u^p\ dx,
\]
we immediately get that
\[
\lim_p\int_\Om \LL^0 \nabla u^p\cdot\nabla u^p\ dx= 
\int_\Om \LL^0 \nabla u\cdot\nabla u\ dx.
\]
A diagonalization process concludes the argument.

\medskip
Consider now, for $u\in \XX$, a sequence $u^\e\in H^1_0(\OR)$ such that $u^\e\rightharpoonup u$ weakly in $L^2(\OR)$.  We revisit Step 2 in the proof of Theorem \ref{thm.main} in Subsection \ref{thmm.main}, taking into account Remark \ref{gen.char.m22=0}. Since, because of that remark,
$N_{22}=0$, \eqref{eq:ident-s} now reads as
$$
\displaystyle \itwot\varphi(x) S(x,y) \cdot \Phi(y)\,dx\,dy=
\displaystyle -\hskip-1cm \sum_{\quad\quad{\text{\tiny $\begin{cases}ijkh\\(k,h)\!\ne\!(2,\!2)\end{cases}$}}}\hskip-.7cm\itwot u_k(x)\, \KK^{\frac12}_{ijkh}(y)\Phi_{ij}(y)\frac{\partial \varphi}{\partial x_h}(x)\,dx\,dy.
$$
The rest of the argument goes through exactly as in Step 2, yielding, in lieu of~\eqref{eq:ga-liminf},
\begin{multline}\label{eq:ga-liminf-bis}
{\liminf_{\e\to 0}}\,{\mathscr I}^\e(u^\e)\ge
\\
\inf\left\{ \itwot\!\!\!\KK(y)(\nabla^\prime_x u(x)+\nabla_yv(y))\cdot(\nabla^\prime_x u(x)+\nabla_yv(y))\,dx\,dy: v\in H^1_\sharp(\Y;\R^2)\right\}
\\
+\iO \!|u|^2\,dx\\ = \inf\!\left\{ \itwot\!\!\!\LL(y)(\nabla^\prime_x u(x)+\nabla_yv(y))\!\cdot\!(\nabla^\prime_x u(x)+\nabla_yv(y))\,dx\,dy: v\in H^1_\sharp(\Y;\R^2)\right\}\\+4\mu_1\int_\Om \det \nabla^\prime_x u \ dx +\iO |u|^2\,dx
\\=\int_\Om \LL^0 \nabla^\prime_x u\!\cdot\! \nabla^\prime_x u\ dx+4\mu_1\int_\Om \det \nabla^\prime_x u \ dx +\iO |u|^2\,dx
\end{multline}
with
$\ds\nabla^\prime_x u:=\nabla_x u- \frac{\partial u_2}{\partial x_2} e_2\otimes e_2.$

It now suffices to remark that, in this specific setting and because $-\lambda_2-\mu_2=\mu_1$,
 the precise expression for $\LL^0$ in the basis $(e_1,e_2)$ is as follows (see~\cite{gutierrez99}):
\begin{align}\label{eq:values}
\ds\LL^0_{1111}=\frac2{\ds\frac 1{\lambda_1+2\mu_1}+\frac 1{\lambda_2+2\mu_2}}=:\bl+2\bm_1,
\\[5mm]\nonumber
\ds\LL^0_{1212}=\LL^0_{1221}=\LL^0_{2112}=\LL^0_{2121}=\frac {2\mu_1\mu_2}{\mu_1+\mu_2}=:\bm_2,
\\[5mm]\nonumber
\ds \LL^0_{1122}=\LL^0_{2211}=\frac{\ds\frac {\lambda_1}{\lambda_1+2\mu_1}+\frac {\lambda_2}{\lambda_2+2\mu_2}}{\ds\frac 1{\lambda_1+2\mu_1}+\frac 1{\lambda_2+2\mu_2}}=-2\mu_1=:\bl,
\\[5mm]\nonumber
\LL^0_{1112}=\LL^0_{1121}=\LL^0_{2111}= \LL^0_{1211}=0,
\\[3mm]\nonumber
 \LL^0_{1222}=\LL^0_{2122}=\LL^0_{2212}=\LL^0_{2221}=0\quad\mbox{and}\quad\LL^0_{2222}=0.
\end{align}
Consequently, recalling Remark \ref{cross-terms},
\begin{multline*}
\int_\Om \LL^0 \nabla_x u\cdot \nabla_x u\ dx=\int_\Om \LL^0 \nabla^\prime_x u\cdot \nabla^\prime_x u\ dx+ 2\ \LL^0_{1122}\int_\Om\frac{\partial u_1}{\partial x_2}\frac{\partial u_2}{\partial x_1}\ dx\\=
\int_\Om \LL^0 \nabla^\prime_x u\cdot \nabla^\prime_x u\ dx+4\mu_1\int_\Om \det \nabla^\prime_x u\ dx.
\end{multline*}
which, in view of \eqref{eq:ga-liminf-bis}, proves the $\Gamma\mbox{-}\liminf$ inequality in the Gutiérrez case.

\begin{remark}
For future reference, we name the material obtained in \eqref{eq:values} the Gutiérrez material and observe that it can be labeled 2D orthorombic since it is invariant under symmetry about the two lines $x_1=0$ and $x_2=0$.\hfill\P
\end{remark}

\subsection{A corollary of Theorem \ref{thm.mainbis}} We conclude this section with a corollary of Theorem \ref{thm.mainbis} that will play an essential role in Section \ref{concl} below.

\begin{corollary}\label{thms.cor}
In the setting of Theorem  \ref{thm.mainbis}, consider the following sequence of functionals
\[
{{\mathscr G}}^\e(v):=\left\{
\begin{array}{cl}
\displaystyle\!\!\iO\!\!\big\{\LL(x/\e) \nabla v\cdot\nabla v+a(x/\e)|v|^2 - b(x/\e) f\cdot v\big\}\,dx, & v\in H^1_0(\OR)
\\[3mm]
\infty, & \mbox{else}.
\end{array}\right.
\]
where $f\in L^2(\OR)$, $a,b\in L^\infty_\sharp(Y_2;\R)$  and  $a(y)\ge\alpha$ a.e. in $\R^2$ for some $\alpha>0$.

Then the results of those theorems still hold true upon replacing ${\mathscr I}^\e$ by ${\mathscr G}^\e$ and ${\mathscr I}^{0}, \mbox{ resp. } {\mathscr I}^{1/2},$ by ${\mathscr G}^{0}, \mbox{ resp. } {\mathscr G}^{1/2}$,  defined by
\[
 {\mathscr G}^0(v):=\left\{\begin{array}{cl}
 \displaystyle\iO \big\{\LL^0 \nabla v\cdot\nabla v+\bar a|v|^2-\bar b f\cdot v\big\}\,dx, & v\in H^1_0(\OR)
 \\[2mm]
 \infty, & \mbox{else},
 \end{array}\right.
\]
and, under the additional assumption that $\LL^0_{2222}=0,$
\[
 {\mathscr G}^{1/2}(v):=\left\{\begin{array}{cl}
 \displaystyle\iO \big\{\LL^0 \nabla v\cdot\nabla v+\bar a|v|^2-\bar b f\cdot v\big\}\,dx, & \ds v\in \XX
 \\[2mm]
 \infty, & \mbox{else},
 \end{array}\right.
\]
where $\bar a:= \int_\Y a(y)\ dy$ (idem for $\bar b$).
\end{corollary}
\begin{proof}
Once again the $\Gamma\mbox{-}\limsup$ inequality is straightforward since, for any $u\in H^1_0(\Om;\R^2)$, or in $\XX$ depending on the case considered, the recovery sequences $\{u^\e\}_\e$ are bounded in $H^1_0(\Om;\R^2)$, so that Rellich's theorem permits one to pass to the limit in the zeroth order term. We obtain
$$
\lim_\e \int_\Om a(x/\e) |u^\e|^2\ dx= \bar a \int_\Om  |u|^2\ dx,\; \lim_\e \int_\Om b(x/\e) f\cdot u^\e\ dx= \bar b \int_\Om  f\cdot v\ dx.
$$
As far as the $\Gamma\mbox{-}\liminf$ inequality is concerned, we still have \eqref{uo-indt-y}, that is that, if a sequence  $\{u^\e\}_\e$ converges weakly in $L^2(\OR)$ to $u\in L^2(\OR)$ and two-scale converges to $u^0(x,y)$, then $u^0$ does not depend on $y$. In other words,  $u^0(x,y)=u(x)$. But then $$
\begin{cases}a(x/\e)u^\e \twosc  a(y) u(x)\\[2mm]
b(x/\e)u^\e \twosc  b(y) u(x)
\end{cases}$$ 
which implies in turn that
$$
\begin{cases}a(x/\e)u^\e \wk \bar a\,u
\\[2mm]
b(x/\e)u^\e \wk \bar b\,u\end{cases}\;\mbox{ weakly in } L^2(\Om;\R^2). 
$$
Consequently passing to the limit in the linear term is immediate while, as far as the quadratic zeroth order term is concerned it suffices to remark that
\begin{multline*}
0\le \liminf_\e\int_\Om a(x/\e)|u^\e-u|^2\ dx= \liminf_\e \int_\Om a(x/\e)|u^\e|^2\ dx-\\[2mm]2 \lim_\e\int_\Om a(x/\e)u^\e\cdot u\ dx + \lim_\e \int_\Om a(x/\e) |u|^2\ dx=\\[2mm] \liminf_\e \int_\Om a(x/\e)|u^\e|^2\ dx-\int_\Om \bar a |u|^2\ dx.
\end{multline*}
With the above inequality at hand, the rest of the proof remains unchanged.
 \end{proof}

\bigskip

\medskip

\section{Elasto-dynamics}\label{concl}
In this last section our  goal is to investigate the impact of Theorem \ref{thm.mainbis} (or rather of Corollary \ref{thms.cor}) on wave propagation with a particular emphasis on the Gutiérrez setting (case (ii) of Theorem \ref{thm.mainbis} with $\theta=1/2$) because of the loss of (strict) strong ellipticity ($\LL^0_{2222}=0$)  demonstrated there. 

\subsection{Convergence of minimizers}\label{sub.minim}
{The main purpose of $\Gamma$-convergence is to ensure the convergence of minimizers. In this respect, consider the $\e$-indexed sequence of functionals 
\[
\displaystyle {\mathscr I}^\e(v)-2\int_\Omega f\cdot v\,dx
\]
where {\em e.g.} $f\in L^2(\OR)$ is a given force load. For a fixed $\e$, we first note in  Lemma \ref{lem.coer} below that $H^1_0$-coercivity holds true.
\begin{lemma}\label{lem.coer}  Under assumptions \eqref{hyp-ph}, \eqref{hyp-lam-mu},
there exists a  a sequence $\{\alpha_\e>0\}_\e$ such that
\[
{\mathscr I}^\e(v) \ge \alpha_\e \int_\Om\left\{|v|^2+|\nabla v|^2\right\}dx,\; v\in H^1_0(\OR).
\]
\end{lemma}
\begin{proof}
Assume that the conclusion does not hold. Then there exists a sequence $v_n\in H^1_0(\OR)$ such that
\be{coer1}
\int_\Om\left\{|v_n|^2+|\nabla v_n|^2\right\}dx=1,
\ee
while
\be{coer2}
{\lim_n{\mathscr I}^\e(v_n)= 0.}
\ee
We can always extend the sequence $v_n$ to $H^1(\R^2;\R^2)$ by setting $v_n\equiv 0$ outside $\Om$. 
In view of \eqref{form.with.K}, \eqref{K-nonneg}, convergence \eqref{coer2} implies in particular that 
\be{coer.L2conv}
v_n\to 0\mbox{ in }L^2(\R^2;\R^2).
\ee
 Further,  the explicit expressions for $\KK_i$ imply that 
\be{coer3.2}
\int_{\R^\e_2} \left[\left(\frac{\partial v^n_1}{\partial x_2}\right)^2+\left(\frac{\partial v^n_2}{\partial x_1}\right)^2+\left(\frac{\partial v^n_1}{\partial x_1}-\frac{\partial v^n_2}{\partial x_2}\right)^2\right]dx \to 0,
\ee
where $\R^\e_2:=\R^2\cap \e Z^\#_2$, while
\be{coer3.1}
\int_{\R^\e_1} \left[\left(\frac{\partial v^n_1}{\partial x_1}+\frac{\partial v^n_2}{\partial x_2}\right)^2
+\left(\frac{\partial v^n_1}{\partial x_2}-\frac{\partial v^n_2}{\partial x_1}\right)^2\right]dx \to 0.
\ee
where $\R^\e_1:=\Om\cap \e Z^\#_1$.

But, as remarked before in  the proof of Step 1 - Case (i) in Subsection \ref{sub.sec.proofbis}, \eqref{coer3.2} is equivalent to stating that $\esym(R^\perp v_n)\to 0$ strongly in $L^2(\R^\e_2;\MS)$.
Because of assumptions \eqref{hyp-ph},    Korn's inequality applies to $\R^\e_2$ and thus   we conclude, with the additional help of \eqref{coer.L2conv}, that $R^\perp v_n\to 0$ strongly in $H^1(\R^\e_2;\R^2)$, hence that 
\be{coer.comv.ph2}
v_n \to 0, \mbox{ strongly in } H^1(\R^\e_2;\R^2).
\ee
Now the determinant is a null Lagrangian, so 
\[
\int_{\R^2} \det \nabla v_n\  dx=0,
\]
hence, in view of \eqref{coer.comv.ph2},
\[
\lim_n \int_{\R^\e_1} \det \nabla v_n\  dx=0.
\]
Subtracting twice that quantity from \eqref{coer3.1}, we obtain
\[
\lim_n\left\{\sum_{i,j=1,2} \int_{\R^\e_1} \left|\frac{\partial{v^n_i}}{\partial x_j}\right|^2\ dx\right\}=0,
\]
which, together with \eqref{coer.comv.ph2}, \eqref{coer.L2conv}, contradicts \eqref{coer1}.
\end{proof}
}

Thanks to  Lemma \ref{lem.coer}, we can now consider the setting of Theorem \ref{thm.mainbis}.
 Take {\em e.g.} $f\in L^2(\OR)$ and consider the minimizer $u^\e\in H^1_0(\OR)$ for the functional
\be{eq.fctle}
\displaystyle {\mathscr I}^\e(v)-2\int_\Omega f\cdot v\,dx.
\ee
That minimizer exists and is unique thanks to the coercivity property in Lemma \ref{lem.coer}, together with the fact that
substitution of $\LL(x/\e)$ by $\KK(x/\e)$ in the expression for ${\mathscr I}^\e$ imparts  convexity on the integrand and, even better, strict convexity in view of the presence of the zeroth order term in the expression for ${\mathscr I}^\e$. 

Remark that $u^\e$ is then the unique solution of the Euler-Lagrange equation associated with the minimization of ${\mathscr I}^\e$ over $H^1_0(\OR)$, that is
\be{eq.elee}
\begin{array}{cl}
-\div(\LL(x/\e) \nabla u^\e)+u^\e=f& \mbox{in }\Omega\\[2mm]
u^\e=0& \mbox{on }\partial \Omega.
\end{array}
\ee
The $\e$-indexed sequence $u_\e$ is  clearly bounded in $L^2(\OR)$ and, thanks to Theorem \ref{thm.mainbis}, we conclude in particular to the $L^2$-weak convergence of  this sequence of  minimizers to the (unique) minimizer $u^0$ in $H^1_0(\OR)$, or $\XX$, depending on the setting, of the $\Gamma$-limit 
\be{eq.fctl0}
\displaystyle {\mathscr I}^0(v)-2\int_\Omega f\cdot v\,dx.
\ee
{In cases (i) or (ii) with $\theta\ne 1/2$, it is  then immediate, through classical variations,  that $u^0$ is the unique $H^1_0(\OR)$-solution of the Euler-Lagrange equation associated with that functional, that is
\be{eq.eleq0not}
\begin{array}{cl}
-\div(\LL^0 \nabla u^0)+u^0=f& \mbox{in }\Omega\\[2mm]
u^0=0& \mbox{on }\partial \Omega.
\end{array}
\ee
In case (ii) with $\theta=1/2$, we need to appeal to the precise values of $\LL^0$ in \eqref{eq:values} and to perform the appropriate variations keeping in mind Remark \ref{cross-terms}.  We easily conclude that $u^0$ satisfies}
\be{eq.eleq0}
\begin{array}{cl}
-\div(\LL^0 \nabla u^0)+u^0=f& \mbox{in }\Omega
\\[2mm]
u^0_1=0, \; u^0_2\,\nu_1=0& \mbox{on }\partial \Omega.
\end{array}
\ee
{where $\LL^0 \nabla u^0$ is {\em a priori} a distribution since $u^0\in \XX$.}

We have thus proved the following

\begin{lemma}\label{lem.ow}
In the setting of Theorem \ref{thm.mainbis}, the unique minimizer for \eqref{eq.fctle} converges weakly in $L^2(\OR)$ to the unique minimizer for \eqref{eq.fctl0}
which further satisfies the Euler-Lagrange equation \eqref{eq.eleq0not}(or \eqref{eq.eleq0} in the Gutiérrez case).
\end{lemma}

\begin{remark}\label{rem.aeps}
Note, for implicit use in the next and final subsection,  that all results (suitably modified) in this subsection 
remain true in the context of Corollary \ref{thms.cor}, that is if the term $\int_\Om |v|^2\ dx$ is replaced by
$\int_\Om a(x/\e)|v|^2\ dx.$
and the linear term $\int_\Om f\cdot v\ dx$ is replaced by $\int_\Om b(x/\e)f\cdot v\ dx$.
\hfill\P\end{remark}

\medskip

\subsection{Wave propagation in the  setting of Theorem \ref{thm.mainbis}}
We now consider a typical problem of elasto-dynamics at fixed $\e$. Consider $(f,g)\in H^1_0(\OR)\times L^2(\OR)$ and the following system for $u^\e(t), \; t\in [0,\infty),$
\be{eq:wee}
\begin{array}{cl}
\ds \rho(x/\e)\frac{\partial^2 u^\e}{\partial t^2}- \div (\LL(x/\e)\nabla u^\e)=0 & \mbox{ in } \Om\times [0,\infty)\\[2mm]
u^\e=0 &\mbox{ on } \partial\Om\times [0,\infty)\\[2mm]
\ds u^\e(0)=f,\quad \frac{\partial u^\e}{\partial t}(0)=g&\mbox{ in }\Om.
\end{array}
\ee
In \eqref{eq:wee}, $\rho(y)$ is the mass density, that is 
$$
\rho(y)=\rho_i  \mbox{\; in } \ZZ_i,\ i=1,2,\quad
 0<\rho_1,\rho_2.
$$
Then, in view of Lemma \ref{lem.coer} and Remark \ref{rem.aeps}, it is classical that this problem has a unique solution $\ds\left(u^\e, \frac{\partial u^\e}{\partial t}\right)\in
C^0([0,\infty); H^1_0(\OR)\times L^2(\OR))$.
Since 
\[
v\mapsto\int_\Om \LL(x/\e)\nabla v\cdot\nabla v \,dx\ge 0,\quad v\in H^1_0(\OR),
\]
we immediately deduce from energy conservation that 
\[
{\frac{\partial u^\e}{\partial t}\mbox{ is bounded in }L^\infty((0,\infty);L^2(\OR)).}
\]
For a subsequence (that we will not relabel), there exists $u^0$ such that
\be{eq.bdonue}
\begin{cases}u^\e \wkst u^0, &\mbox{weakly-* in } W^{1,\infty}_{\rm loc}([0,\infty);L^2(\OR))
\\[2mm]
\ds\frac{\partial u^\e}{\partial t}\wkst \frac{\partial u^0}{\partial t}, &\mbox{weakly-* in } L^{\infty}([0,\infty);L^2(\OR)).
\end{cases}
\ee
 Furthermore, the Laplace transform 
\[
\hat u^\e(\ell):=\int_0^\infty u^\e(t)e^{-\ell t}\ dt, \;\ell>0,
\]
of $u^\e$ satisfies
\[
\begin{array}{cl}
\ell^2 \rho(x/\e)\hat u^\e(\ell)- \div(\LL(x/\e)\nabla \hat u^\e(\ell))= \rho(x/\e)(\ell f+g) & \mbox{ in }\Om
\\[2mm]
\hat u^\e(\ell)=0&\mbox{ on } \partial\Om.
\end{array}
\]
Recalling \eqref{eq.elee}, we infer that $\hat u^\e(\ell)$ is the unique $H^1_0(\OR)$-mimimizer
of 
\[
\displaystyle\iO\big\{\LL(x/\e) \nabla v\cdot\nabla v+\ell^2 \rho(x/\e)|v|^2- 2\rho(x/\e)(\ell f+g)\cdot v\big\}\,dx.
\]
But then, applying Lemma \ref{lem.ow} (and Remark \ref{rem.aeps}), we conclude that, at least in the settings validated in  Theorem \ref{thm.mainbis} and Corollary \ref{thms.cor}, 
\[
\hat u^\e(\ell) \wk \check u^0(\ell), \mbox{ weakly in } L^2\OR),
\]
where $\check u^0(\ell)$ is the unique $H^1_0$ (or $\XX$ in case (ii))-minimizer of
\[
\displaystyle\iO\big\{\LL^0 \nabla v\cdot\nabla v+\ell^2\bar \rho |v|^2- 2\bar\rho\,(\ell f+g)\cdot v\big\}\,dx
\]
with 
\be{eq.barr}
\bar \rho:= \theta\rho_1+(1-\theta) \rho_2.
\ee

Now, in view of \eqref{eq.bdonue}, 
\be{eq.chha}
\check u^0(\ell)=\hat u^0(\ell),\; \ell>0.
\ee

In case (i) or case (ii) with $\theta\ne 1/2$, the system of elasto-dynamics 
\be{eq:we0}
\begin{array}{cl}
\ds\bar\rho\,\frac{\partial^2 \bar u^0}{\partial t^2}- \div (\LL^0\nabla \bar u^0)=0 & \mbox{ in } \Om\times [0,\infty)\\[2mm]
\bar u^0=0 &\mbox{ on } \partial\Om\times [0,\infty)\\[2mm]
\ds \bar u^0(0)=f,\quad \frac{\partial \bar u^0}{\partial t}(0)=g&\mbox{ in }\Om
\end{array}
\ee
has a unique solution $\ds \left(\bar u^0, \frac{\partial \bar u^0}{\partial t}\right)$ in $C^0([0,\infty); H^1_0(\OR)\times L^2(\OR))$, as can be easily checked by {\em e.g.} extending all functions to $0$ ouside $\Om$ and recalling Remark \ref{sse.l0} which implies coercivity of 
\[
\int_{\R^2}\left\{\LL^0 \nabla v\cdot \nabla v+ \bar\rho|v|^2\right\}dx
\]
over $H^1_0(\R^2;\R^2)$.

The system 
\be{eq:we0G}
\begin{array}{cl}
\ds\bar\rho\,\frac{\partial^2 \bar u^0}{\partial t^2}- \div (\LL^0\nabla \bar u^0)=0 & \mbox{ in } \Om\times [0,\infty)\\[2mm]
\bar u^0_1=0,\ \bar u^0_2\,\nu_1=0 &\mbox{ on } \partial\Om\times [0,\infty)\\[2mm]
\ds \bar u^0(0)=f,\quad \frac{\partial \bar u^0}{\partial t}(0)=g&\mbox{ in }\Om
\end{array}
\ee
\noindent
also possesses a unique solution in case (ii) with $\theta=1/2$, which is admittedly less classical.

The result can be obtained through various methods. For example, one can remark that the operator
\[
\begin{array}{c}
\ds \L^0:= \left(\begin{matrix}0&1
\\
\div(\LL^0\nabla \cdot) & 0\end{matrix}
\right)
\end{array}
\]
with domain 
$D(\L^0)\!:=\! \left\{(v,w)\in\XX\times L^2(\OR) :w\in \XX,\,\div (\LL^0\nabla v) \in L^2(\OR)\right\}$
is  skew self-adjoint on the Hilbert space $\XX\times L^2(\OR)$. 

{To that end, one should endow $\XX$ with the inner product
\[
\langle\!\langle u,v\rangle\!\rangle_\XX:= \int_\Om \LL^0\nabla u\cdot\nabla v\ dx.
\]
Because $\Om$ is bounded and $C^\infty_c(\OR)$ is dense in $\XX$, it is easily checked that this new inner product generates a norm which is equivalent to that, denoted here by $\|\cdot\|_\XX$, associated with the inner product $\langle\cdot,\cdot\rangle_\XX$ defined in Remark \ref{structureX}. Indeed, take $\ph\in C^\infty_c(\OR)$. Then, with the help of \eqref{eq:values},
\begin{multline*}
\ds\int_\Om \LL^0\nabla\ph\cdot\nabla\ph\,dx=
\int_{\R^2} \left\{(\bl+2\bm_1)\,\xi_1^2|\hat\ph_1|^2+2(\bl+\bm_2)\,\xi_1\xi_2\,\mbox{Re}(\hat\ph_1\bar{\hat\ph}_2)\right.
\\[2mm]
\left.+\;\bm_2\big(\xi_2^2|\hat\ph_1|^2+\xi_1^2|\hat\ph_2|^2\big)\right\}d\xi=
\\[2mm]
\ds\int_{\R^2} \left\{(\bl+2\bm_1)\,\xi_1^2|\hat\ph_1|^2+\bar L\begin{pmatrix}\xi_2\hat{\ph}_1 \\ \xi_1\hat{\ph}_2\end{pmatrix}\cdot
\overline{\begin{pmatrix}\xi_2\hat{\ph}_1 \\ \xi_1\hat{\ph}_2\end{pmatrix}}\,\right\}d\xi
\end{multline*}
where $\bar L:=\left(\begin{smallmatrix}\bm_2&\bl+\bm_2\\\bl+\bm_2&\bm_2\end{smallmatrix}\right).$
But, in view of \eqref{eq:values}, the quantities $\bl+2\bm_1$, $-\bl$, $\bm_2$, $\bl+2\bm_2$ are positive, so that the integrand in the right hand-side of the second equality above is bounded below by 
\begin{multline*}
c\! \int_{\R^2}\!\!\big(\xi_1^2|\hat\ph_1|^2+\xi_2^2|\hat\ph_1|^2+ \xi_1^2|\hat\ph_2|^2\big)d\xi
=\!c\!\left(\!\|\nabla \ph_1\|_{L^2(\OR)}^2\!+\!\left\|\frac{\partial \ph_2}{\partial x_1}\right\|_{L^2(\Om)}^2\!\right)\!\ge \!c'_\Om\|\ph\|^2_{\XX}
\end{multline*}
for some $c>0$ and some $\Om$-dependent $c'_\Om>0$. Note that the  inequality above holds true precisely because $\Om$ is bounded.
}

It then suffices to apply Stone's theorem for unitary groups of operators (see {\em e.g.} \cite[Chapter IX-9]{yosida}).

\medskip

Thus, in both cases (i) and (ii), the Laplace transform for any $\ell>0$ of the unique solution $\bar u^0$ to \eqref{eq:we0}   is precisely $\check u^0(\ell)$ which in turn, thanks to \eqref{eq.chha}, is the Laplace transform $\hat u^0(\ell)$ of $u^0$ given through convergences \eqref{eq.bdonue}. Thus $\bar u^0=u^0$ and, in view of the uniqueness of the limit function $u^0$, there is no need to extract subsequences.

We have proved that the following theorem holds true:
\begin{theorem}\label{thm.elastod}
In the setting of Theorem \ref{thm.mainbis} and given $$(f,g)\in H^1_0(\OR)\times L^2(\OR)$$ as initial conditions, the unique solution $$\ds\left(u^\e, \frac{\partial u^\e}{\partial t}\right)\in C^0([0,\infty); H^1_0(\OR)\times L^2(\OR))$$ to system \eqref{eq:wee} converges weakly-* in $$W^{1,\infty}_{\rm loc}((0,\infty);L^2(\OR))\times L^\infty((0,\infty);L^2(\OR))$$ to $$\left(u^0, \frac{\partial u^0}{\partial t}\right)\in C^0([0,\infty); H^1_0(\OR)(\mbox{resp. }\XX)\times L^2(\OR)),$$ the unique solution  to system \eqref{eq:we0} with $\bar \rho$ defined by \eqref{eq.barr} in cases (i) or (ii) with $\theta\ne 1/2$ (resp. case (ii) with $\theta=1/2$).
\end{theorem} 

\vskip.5cm

We conclude this study in the company of Gutiérrez. In that case, we know that $\LL^0_{2222}=0$.
\par
As far as the two-phase laminate at fixed $\ep$ is concerned, the elasto-dynamic problem \eqref{eq:wee} cannot have a non-zero plane wave solution of the type
\[
u^\e(t,x)=F^\e(t+\xi^\e\!\cdot x)+G^\e(t-\xi^\e\!\cdot x)\quad\mbox{for some }\xi^\e\in\R^2\setminus\{0\},
\]
because of  the Dirichlet boundary condition satisfied by $u^\e(t,\cdot)$ on $\partial\Om$, and this whatever the initial conditions $f$ and $g$ might be.  The same applies to  the elasto-dynamic problem \eqref{eq:we0} associated with the homogenized material in  cases (i) and (ii) with $\theta\neq1/2$.     
\par
On the contrary, in the Guti\'errez setting (ii) with $\theta=1/2$, starting {\em e.g.} with the initial conditions
\be{square.ic}
f(x)=2\sin x_1,\ g(x)=0\quad\mbox{for }x\in\Om=(0,\pi)^2,
\ee
it is easy to check that the function $u^0=(0,u^0_2)\in C^0([0,\infty);\XX)$, with
\be{square.sol}
u^0_2(t,x_1)=\sin\Big(\textstyle{\sqrt{\frac{\LL^0_{1212}}{\bar\rho}}}\,t+x_1\Big)-\sin\Big(\textstyle{\sqrt{\frac{\LL^0_{1212}}{\bar\rho}}}\,t-x_1\Big)\quad\mbox{for } t\in[0,\infty),\ x_1\in(0,\pi)
\ee
(where $\LL^0_{1212}=2\mu_1\mu_2/(\mu_1+\mu_2)$), is a transverse plane wave solution to the homogenized problem \eqref{eq:we0}. This is so because  the space $\XX$ replaces $H^1_0(\Om;\R^2)$ and thus does not require any boundary condition for $u_2$ on the horizontal sides of the  square.

Forgetting now about boundary conditions, we would like to investigate the kind of plane waves that the Gutiérrez material can withstand on the whole plane. To this effect, we find it more convenient to write a (2D orthorombic) stress-strain relation for that material in the form
\[
\begin{cases}
\sigma_{11}=\bl\,\div u + 2\bm_1\,\esym_{11}(u)\\[2mm]
\sigma_{12}=2\bm_2\,\esym_{12}(u)\\[2mm]
\sigma_{22}=\bl\,\div u +2\bm_3\,\esym_{22}(u),
\end{cases}
\]
where $\bl$, $\bm_1$, $\bm_2$ have been defined in \eqref{eq:values} and $\bm_3:= \mu_1$,
so that, in particular, $\bl+2\bm_3=0$, while $\bl+2\bm_1,\bl+2\bm_2>0$ as well as the $\bm_i$'s for $i=1,2,3$. 
We seek a plane wave solution of the form
\[
u(t,x)= e^{i(k\cdot x-\omega t)}\,\eta,\quad\eta, k\in \R^2,\ |k|=1,\ \omega\in\R,
\]
of equation \eqref{eq:we0}.
After some algebra, this amounts to finding the eigenvalues, {\it i.e.,} the $(\bar\rho\,\omega^2)$'s,  of the
symmetric matrix
\[
A(k):=\left(\begin{matrix}(\bl+2\bm_1)k_1^2+\bm_2 k^2_2&(\bl+\bm_2)k_1k_2
\\[2mm]
(\bl+\bm_2)k_1k_2&(\bl+2\bm_3)k_2^2+\bm_2 k^2_1\end{matrix}\right),
\]
in which case the corresponding eigenvectors are the directions of propagation, {\it i.e.,} the $\eta$'s.
In our setting, and even in the case where $\bl+2\bm_3>0$ (which could be obtained by increasing the value of $\bl$),  it is an easy task to check that, provided that 
\[
(\bl+2\bm_1)(\bl+2\bm_3)-\bl(\bl+2\bm_2)\ge 0,
\]
 the two eigenvalues of $A(k)$
are always nonnegative. Further,  the case 
\be{eq.comp0}\bl+2\bm_3=0,
\ee 
which is directly in the spirit of the Cauchy material (see \eqref{zero.ws}),  is the only case for which one of the eigenvalues can be zero.  This happens for, and only for $k_1=0$. There, the plane wave is a transversal (shear) wave oscillating in  the direction $e_2$ and propagating in the direction $e_1$.

So, in essence, the singular behavior of the Gutiérrez material resides both in the possibility of propagating plane waves on a bounded domain with Dirichlet boundary conditions -- as demonstrated through \eqref{square.ic}, \eqref{square.sol} -- and in the existence  of one, and only one direction in which longitudinal waves cannot propagate, namely the direction of lamination.

It is somewhat tempting to view the Gutiérrez material as a contemporary version of {labile} aether, that is of an elastic  material that does not support longitudinal waves as demonstrated by \eqref{eq.comp0}. See \cite[Chapter 5]{whittaker} for a description of Cauchy's,  Green's, and Thomson's attempts in this direction. However, our material merely seems to represent its putative manifestation in a 2D orthorombic crystal because it only prevents the existence of longitudinal waves in one direction.  

But ultimately the main difference is  that  aether is of course three-dimensional, a setting for which a similar analysis is wanting at present.  Gutiérrez has also produced in \cite{gutierrez99}, through multiple layering, a 3D material that  loses strict strong ellipticity. It is our unsubstantiated hope that the present analysis can be extended to that case as well.

\section*{acknowledgements}
{\small G.F. acknowledges the support of the  the National Science Fundation Grant DMS-1615839.
The authors also thank Giovanni Leoni for his help in establishing Remark \ref{structureX}, Patrick Gérard for fruitful insights into propagation in the absence of coercivity and Lev Truskinovsky for introducing us to the fascinating history of the elastic aether.}

\bibliographystyle{plain}
\bibliography{gaf.bib}

\end{document}